\documentclass[final]{siamart190516}
\usepackage{graphicx,psfrag,amsmath,amsfonts,verbatim, mathtools}
\usepackage{enumerate}
\usepackage{enumitem}

\numberwithin{equation}{section}
\usepackage{hyperref}

\usepackage{subfigure}
\usepackage{epstopdf}
\usepackage{float}

\usepackage{multirow}
\usepackage{booktabs}
\usepackage{makecell}

\usepackage{bm}
\usepackage{amssymb}
\usepackage{ntheorem}
\usepackage{framed}
\usepackage{color}
\definecolor{gray}{rgb}{0.6,0.6,0.6}

\newtheorem{defn}{Definition}[section]
\newtheorem{thm}{Theorem}[section]

\newtheorem{lm}{Lemma}[section]
\newtheorem{assume}{Assumption}[section]
\newtheorem{remark}{Remark}[section]
\newtheorem{cor}{Corollary}[section]

\newcommand\Var {\mathbb{V}}

\usepackage{xcolor}
\usepackage{framed}
\usepackage{lipsum}
\colorlet{shadecolor}{blue!20}

\usepackage{macros}

\begin{document}

\title{Sample Complexity of Sample Average Approximation for Conditional Stochastic Optimization}
\author{
	Yifan Hu\thanks{Department of Industrial and Enterprise Systems Engineering (ISE), University of Illinois at Urbana-Champaign (UIUC), Urbana, IL (\email{yifanhu3@illinois.edu},
	\email{xinchen@illinois.edu},
	\email{niaohe@illinois.edu}.)}
	\and
	Xin Chen\footnotemark[1]
	\and
	Niao He\footnotemark[1]
}
\maketitle

\begin{abstract}
In this paper, we study a class of stochastic optimization problems, referred to as the \emph{Conditional Stochastic Optimization} (CSO), in the form of $\min_{x \in \mathcal{X}} \EE_{\xi}f_\xi\Big({\EE_{\eta|\xi}[g_\eta(x,\xi)]}\Big)$, {which finds a wide spectrum of applications including portfolio selection, reinforcement learning, robust learning, causal inference and so on. Assuming availability of samples from the distribution $\PP(\xi)$ and samples from the conditional distribution $\PP(\eta|\xi)$,} we establish the sample complexity of the sample average approximation (SAA) for CSO, under a variety of structural assumptions, such as Lipschitz continuity, smoothness, and error bound conditions. We show that the total sample complexity improves from $\cO(d/\eps^4)$ to $\cO(d/\eps^3)$ when assuming smoothness of the outer function, and further to $\cO(1/\eps^2)$ when the empirical function satisfies the quadratic growth condition. We also establish the sample complexity of a modified SAA, when $\xi$ and $\eta$ are independent. Several numerical experiments further support our theoretical findings. 
\end{abstract}



\section{Introduction}
Decision-making in the presence of uncertainty  has been a fundamental and long-standing challenge in many fields of science and engineering. In recent years, extensive research efforts have been devoted to the design and theory of efficient algorithms for solving the classical stochastic optimization (SO) in the form of
\begin{equation}
\label{eq:SO}
\min_{x \in \mathcal{X}}\; F(x):=\EE_{\xi}[f(x,\xi)], 
\end{equation} 
ranging from convex to non-convex objectives, from first-order to second-order methods, and from sub-linear to linear convergent algorithms, e.g., see ~\cite{BCN18} and references therein for a comprehensive survey. Here $\mathcal{X}\subseteq \RR^d$ is the decision set and $f(x,\xi)$ is some cost function dependent on the random vector $\xi$. In general, (\ref{eq:SO}) cannot be computed analytically or solved exactly, even when the underlying distribution of the random vector $\xi$ is known, and one has to resort to Monte Carlo sampling techniques. 

An important Monte Carlo method -- the sample average approximation (SAA, a.k.a., the empirical risk minimization in machine learning community) -- is widely used to solve (\ref{eq:SO}), assuming availability of samples from the underlying distribution. SAA works by solving the approximation of the original problem: 
\begin{equation}\label{eq:SO-SAA}
\min_{x \in \mathcal{X}}\; \hat{F}_n(x):= \frac{1}{n}\sum_{i=1}^n f(x,\xi_i),
\end{equation}
where $\xi_1,\ldots,\xi_n$ are i.i.d. samples generated from the distribution of $\xi$. Note that $\hat F_n(x)$ converges pointwise to $F(x)$ with probability $1$ as $n$ goes to infinity. Finite-sample convergence of SAA for SO has been well established.  The seminal work by~\cite{kleywegt2002sample} proved that for general Lipschitz continuous objectives, SAA requires  a sample complexity  of $\cO(d/\eps^2)$ to obtain an $\epsilon$-optimal solution to the stochastic optimization problem.  \cite{shalev2010learnability} proved that for strongly convex and Lipschitz continuous objectives,  the sample complexity of SAA is $\cO(1/\eps)$. Detailed results can be found in the books, e.g., \cite{shapiro2009lectures} and \cite{shalev2014}. 

More generally, SAA is also a popular computational tool for solving multi-stage stochastic programming (MSP) problems. In its general form, a MSP  finds a sequence of decisions $\{x_t\}_{t=0}^T$ that minimizes the nested expectation in the following form:
\begin{equation}\label{eq:multistage}
\min_{x_0\in \mathcal{X}_0} f_0(x_0)+\EE_{\xi_1}\bigg[\inf_{x_1
}\; f_1(x_1,\xi_1) + \EE_{\xi_2|\xi_1}\Big[ \cdots+ \EE_{\xi_T|\xi_{1:T-1}}\big[\inf_{x_T
}\; f_T(x_T,\xi_T)\big]\Big] \bigg],
\end{equation}
where $T$ is the number of decision periods, $\xi_1,\ldots, \xi_T$ can be considered as a random process, and the decision $x_t$ is a function of the history of the process up to time $t$. Similarly, the SAA approach works by first generating a large scenario tree with conditional sampling and then processing with stage-based or scenario-based decomposition methods \cite{pereira1991multi,rockafellar1991scenarios,ruszczynski1997decomposition}. When extended to the multi-stage case, the finite sample analysis indicates that the total number of samples, or scenarios, to achieve an $\epsilon$-optimal solution to the original problem (\ref{eq:multistage}) grows exponentially as the number of stages increases~\cite{shapiro2005complexity,shapiro2009lectures}. In particular, for general three-stage stochastic problems,  the sample complexity of SAA cannot be smaller than $\cO(d^2/\eps^4)$; this holds true even if the cost functions in all stages are linear and the random vectors are stage-wise independent as discussed in~\cite{shapiro2006complexity}. 

In this paper, we study an intermediate class of problems, referred to as the \emph{Conditional Stochastic Optimization} (CSO), that sits in between the classical SO and the MSP. The problem of interest takes the following general form:
\begin{equation}
\label{pro:ori}
    \min_{x \in \mathcal{X}}\; F(x):=\EE_{\xi}\Big[f_\xi\Big({\EE_{\eta|\xi}[g_\eta(x,\xi)]}\Big)\Big].
\end{equation}
Here $\mathcal{X}$ is the domain of the decision variable $x\in\RR^d$; $f_\xi(\cdot):\RR^k\to\RR$ is a continuous cost function dependent on the random vector $\xi$, and $g_\eta(\cdot,\xi):\RR^d\to\RR^k$  is a vector-valued continuous cost function dependent on both random vectors $\xi$ and $\eta$. The inner expectation is with respect to $\eta$ given $\xi$, and the outer expectation is with respect to $\xi$. Similar to the classical stochastic optimization, we do not assume any knowledge on the underlying distribution of $\PP(\xi)$ nor the conditional distribution $\PP(\eta|\xi)$.  Instead, we assume availability of  samples  from the distribution $\PP(\xi)$ and samples from the conditional distribution $\PP(\eta|\xi)$ for any given $\xi$.

CSO is more general than the classical stochastic optimization as it captures dynamic randomness and involves conditional expectation. It takes the SO as a special case when $g_\eta(x,\xi)$ is an identical function. On the other hand, it is less complicated than the MSP (in particular the three-stage case with $T=3$) as it seeks for a static decision and does not subject to non-anticipativity constraints. 

The goal of this paper is to analyze the sample complexity of SAA for solving CSO, which can be constructed as follows based on \emph{conditional sampling}:
\begin{equation}
\label{pro:cosaa0}
   \min_{x \in \mathcal{X}}\; \hat F_{nm}(x):= \frac{1}{n}\sum_{i=1}^n f_{\xi_i}\bigg(\frac{1}{m}\sum_{j=1}^m g_{\eta_{ij}}(x, \xi_{i})\bigg),
\end{equation}
where  $\{\xi_i\}_{i=1}^n$ are i.i.d. samples generated from $\PP(\xi)$ and $\{\eta_{ij}\}_{j=1}^m$ are i.i.d. samples generated from the conditional distribution $\PP(\eta | \xi_i)$ for a given outer sample $\xi_i$.  We would like to examine the total number of samples $T=nm+n$ required for SAA~(\ref{pro:cosaa0}) to achieve an $\epsilon$-optimal solution to the original CSO problem~(\ref{pro:ori}).

We also consider a special case of the CSO problem (\ref{pro:ori}), when the random vectors $\xi$ and $\eta$ are independent:
\begin{equation}
\label{eq:imain}
    \min_{x \in \mathcal{X}}\; F(x):=\EE_{\xi}\Big[f_\xi\Big({\EE_{\eta}[g_\eta(x,\xi)]}\Big)\Big].
\end{equation}
One could still approximate (\ref{eq:imain}) by the SAA~(\ref{pro:cosaa0}), mimicking the conditional sampling scheme and using different samples $\{\eta_{i1},\ldots, \eta_{im}\}$ from the distribution of $\eta$ for each $\xi_i$. However, since the inner expectation is no longer a conditional expectation, there is no necessity to estimate the inner expectation with different realizations of $\eta$ for each $\xi_i$.  Hence, an alternative way to approximate  (\ref{eq:imain}) is through a modified SAA: 
\beq{eq:iSAA}
\min_{x \in \mathcal{X}}\; \hat F_{nm}(x):= \frac{1}{n}\sum_{i=1}^n f_{\xi_i}\bigg(\frac{1}{m}\sum_{j=1}^m g_{\eta_{j}}(x, \xi_{i})\bigg).
\eeq
where  $\{\xi_i\}_{i=1}^n$ are i.i.d. samples generated from the distribution of $\xi$ and $\{\eta_{j}\}_{j=1}^m$ are i.i.d. samples generated from the distribution of  $\eta$.
As a result, the component functions $f_{\xi_i}\Big(\frac{1}{m}\sum_{j=1}^m g_{\eta_{j}}(x, \xi_{i})\Big)$, $i=1,\ldots, n$ become dependent since they share the same $\{\eta_j\}_{j=1}^m$, making it very different from ~(\ref{pro:cosaa0}). In this case, the total number of samples becomes $T = n+m$. We refer to this sampling scheme as \emph{independent sampling}.

\subsection{Motivating Applications}
Notably, CSO can be used to model a variety of applications, including portfolio selection~\cite{hong2017kernel}, robust supervised learning~\cite{pmlr-v54-dai17a}, reinforcement learning~\cite{pmlr-v54-dai17a,pmlr-v80-dai18c}, personalized medical treatment~\cite{ikko2018}, instrumental variable regression~\cite{muandet2019dual}, and so on. We discuss some of these examples in details below. 

\paragraph{Robust Supervised Learning} 
Incorporation of priors on invariance and robustness into the supervised learning procedures is  crucial for computer vision and speech recognition~\cite{niyogi1998incorporating,bhagoji2018enhancing}. Taking image classification as an example, we would like to build a classifier that is both accurate and invariant to certain kinds of data transformation, such as rotation or perturbation. Let $\xi_1=(a_1,b_1),\cdots, \xi_n=(a_n,b_n)$ be a set of input data, where $a_i$ is the feature vector and $b_i$ is the label.  A plausible way to achieve such consistency is to consider the class of robust linear classifiers, say $f(x,x_0,\xi)=\EE_{\eta|\xi\sim\mu(\sigma(a))}[x^T\eta+x_0]$ for  given image data $\xi$, by averaging the prediction over all possible transformations $\sigma(a)$, and then finding the best fit by minimizing the expected risk: 
\begin{equation*}
\min_{(x,x_0)} \; \EE_{\xi=(a,b)}\Big[\ell\big(b,\EE_{\eta|\xi}[\eta^Tx+x_0]\big)\Big] +\frac{\nu}{2}\|x\|_2^2.
\end{equation*}
Here $\ell(\cdot,\cdot)$ is some loss function, $\nu>0$ is a regularization parameter, and $\mu(\cdot)$ is a given distribution (e.g., uniform) over the transformations.  Clearly, such problems belong to the category of CSO. 

\paragraph{Reinforcement Learning}
Policy evaluation is a fundamental task in Markov decision processes and reinforcement learning. Consider a discounted Markov
decision process characterized by the tuple $\cM: =(\cS,\cA,P,r,\gamma)$, where $\cS$ is a finite state space, $\cA$ is a finite action
space, $P(s,a,s')$ represents the (unknown)
state transition probability from state $s$ to $s'$ given action
$a$, $r(s,a):{\cal S}\times {\cal A}\to\RR$ is a reward function, and $\gamma \in (0,1)$ is a discount factor. Given a stochastic policy $\pi(a|s)$, the goal of the policy evaluation is to estimate the value function $V^\pi(s):={\mathbb E} \left[ {\left. \sum_{k=0}^\infty {\gamma^k r(s_k,a_k)}\right|s_0=s}\right]$ under the policy. It is well-known that $V^\pi(\cdot)$ is a fixed point of the Bellman equation~\cite{bertsekas2005dynamic},
\begin{equation*}
V^\pi(s) = \EE_{s^\prime|a,s} [r(s,a)+\gamma V^\pi(s^\prime)].
\end{equation*}
To estimate the value function $V^\pi(s)$, one could resort to minimizing the mean squared Bellman error \cite{sutton2008, pmlr-v54-dai17a}, namely:
\begin{equation*}
\min_{V(\cdot):\cS\to\RR}\; \EE_{s\sim\mu(\cdot),a\sim\pi(\cdot|s)}\big[\big(r(s,a)-\EE_{s^\prime|a,s}[V(s)-\gamma V(s^\prime)]\big)^2\big].
\end{equation*}
Here $\mu(\cdot)$ is the stationary distribution. This minimization problem can be viewed as a special case of CSO. Recently, \cite{pmlr-v80-dai18c} showed that finding the optimal policy can also be formulated into an optimization problem in a similar form by exploiting the smoothed Bellman optimality equation. Again, the resulting problem falls under the category of CSO. 

\paragraph{Uplift Modeling} Uplift modelling aims at estimating  individual treatment effects, and it has been widely studied in causal inference literature and used for personalized medicine treatment and targeted marketing ~\cite{jaskowski2012uplift,ikko2018}. In an individual uplift model, the goal is to estimate the effect of a treatment on an individual with feature vector $x$, which could be represented by  $u(x):=\EE[y|x,t=1]-\EE[y|x,t=-1]$. Here $t\in\{\pm 1\}$ represents whether a treatment has been given to an individual, and $y\in\cY\subseteq\RR$ represents the outcome. In practice, obtaining joint labels $(y,t)$ can be difficult, whereas obtaining one label (either $t$ or $y$) of the individual is relatively easier. \cite{ikko2018} considered an individual uplift model that assumes availability of only one label  from the joint labels, and estimates the unknown label with 
$
p(y | x)=\sum_{t=\{\pm 1\}} p(y | x, t) p(t | x)
$.
They showed that the individual uplift $u(x)$ is equivalent to the optimal solution to the following least-squares problem:
\begin{equation*}
\min_{u\in L^2(p)}\; \EE_{x\sim p(x)} \left[(\EE_{w|x}[w]\cdot u(x)-2\EE_{z|x}[z])^2\right],    
\end{equation*}
where $L^2(p) = \{f:\mathcal{X}\rightarrow \RR|\; \EE_{x\sim p(x)} [f(x)^2] <\infty\}$ is a function space, and $w$ and $z$ are two auxiliary random variables, whose conditional density are given by
$p\left(z=z_{0} | x\right)=\frac{1}{2} p\left(y=z_{0} | x\right)+\frac{1}{2} p\left(y=-z_{0} | x\right)$, $p\left(w=w_{0} | x\right)=\frac{1}{2} p\left(t=w_{0} | x\right)+\frac{1}{2} p\left(t=-w_{0} | x\right)$.
If we further restrict $u(\cdot)$ to a finite dimensional parameterization, then the above problem becomes a special case of CSO.  

For these applications, there are many settings in which samples can be generated according to our assumptions. For instance, in robust supervised learning and uplift modeling, there are multiple samples from $\PP(\eta|\xi)$ available for any given $\xi$.

\subsection{Related Work} 
A closely related class of problems, called \emph{stochastic composition optimization}, has been extensively studied in the literature; see, e.g.,~\cite{Yermolyev1971,polyak1978minimization,ermoliev2013sample,wang2017stochastic}, to name just a few. This class of problems takes the following form: 
\beq{eq:wang}
\min_{x \in \mathcal{X}}\;  \mathbf{f}\circ \mathbf{g}(x):=\EE_{\xi}\Big[f_\xi\Big({\EE_{\eta}[g_\eta(x)]}\Big)\Big],
\eeq
where $\mathbf{f}(u):=\EE_{\xi}[f_\xi(u)]$, and $\mathbf{g}(x):=\EE_{\eta}[g_\eta(x)]$. Although the two problems, (\ref{eq:wang}) and (\ref{pro:ori}), share some similarities in that both objectives are represented by nested expectations, they are fundamentally different in two aspects: i) the inner randomness $\eta$ in \rf{pro:ori} is conditionally dependent on the outer randomness $\xi$, while the inner expectation in \eqref{eq:wang}  is taken over the marginal distribution of $\eta$; ii) the inner random function $g_{\eta}(x,\xi)$ in (\ref{pro:ori}) depends on both $\xi$ and $\eta$. As a result, unlike (\ref{eq:wang}), the CSO problem \eqref{pro:ori} cannot be formulated as a composition of two deterministic functions due to the dependence between the inner and outer function. Another key distinction from (\ref{eq:wang}) is that we assume availability of samples from the distributions $\PP(\xi)$ and $\PP(\eta|\xi)$, rather than samples from the joint distribution $\PP(\xi,\eta)$. These two distinctions further lead to a drastic difference in the SAA construction and the sample complexity analysis of these two types of problems, as we will show in the rest of the paper. 

When solving either (\ref{eq:wang}) or (\ref{pro:ori}), most of the existing work is devoted to developing stochastic oracle-based algorithms and their convergence analysis for solving these problems.  Related work includes  two-timescale  \cite{polyak1978minimization,Yermolyev1971,wang2017stochastic,wang2016accelerating} and single-timescale \cite{ghadimi2018single} stochastic approximation algorithms  for solving the problem (\ref{eq:wang}), variance-reduced algorithms for solving the SAA counterpart of (\ref{eq:wang}) \cite{pmlr-v54-lian17a, AAAI1817203, shen2018asynchronous}, and a primal-dual functional stochastic approximation algorithm for solving the problem (\ref{pro:ori})~\cite{pmlr-v54-dai17a}. These methods usually require convexity of the objective in order to obtain an $\eps$-optimal solution. Our work differs from the ones listed above in that we mainly focus on establishing the sample complexity of SAA itself, rather than designing efficient algorithms to solve the resulting SAA. 

We point out that our paper has the same strain as a series of papers~\cite{kleywegt2002sample,shapiro2005complexity,shapiro2006complexity,pagnoncelli2009sample,ermoliev2013sample, dentcheva2017statistical,bertsimas2018robust,liu2019sample}, centered at the sample average approximation approach for stochastic programs. In particular,~\cite{dentcheva2017statistical} derived a central limit theorem result for the SAA of the stochastic composition optimization problem (\ref{eq:wang}) and ~\cite{ermoliev2013sample} established the rate of convergence. Despite these developments, the study of the basic SAA approach and its finite sample complexity analysis remains unexplored for solving the general CSO problem (\ref{pro:ori}) and even the special case (\ref{eq:imain}). We aim to close this gap in this paper. 

\subsection{Contributions}
In this paper, we formally analyze the sample complexity of the corresponding SAA approach for solving CSO. Our contributions are summarized as follows and in Table~\ref{tab:result}. 
\begin{enumerate}[leftmargin=*]
    \item[(a)]  We establish the first sample complexity results of the SAA in (\ref{pro:cosaa0}) for the CSO problem (\ref{pro:ori}) under several structural assumptions:
    \begin{itemize}
        \item[(i)] Both $f_\xi$ and $g_\eta$ are Lipschitz continuous;
        \item[(ii)] In addition to (i), $f_\xi$ is Lipschitz smooth;
        \item[(iii)] In addition to (i), the empirical function satisfies the H\"olderian error bound condition;
        \item[(iv)] In addition to (i), $f_\xi$ is Lipschitz smooth and the empirical function satisfies the H\"olderian error bound condition.
    \end{itemize}
None of these assumptions require convexity\footnote{ However, when solving the SAA problem itself, convexity conditions are often necessary for obtaining a global minimizer.} of the underlying objective function. 
Note that the H\"olderian error bound (HEB) condition~\cite{Bolte2017}, which includes the quadratic growth (QG) condition~\cite{karimi2016linear} as a special case, is a much weaker assumption than strong convexity, and holds for many nonconvex problems in machine learning applications~\cite{pmlr-v80-charles18a}. We show that, for general Lipschitz continuous problems, the sample complexity of SAA improves from $\cO(d/\epsilon^4)$  to $\cO(d/\epsilon^3)$  when assuming smoothness; for problems satisfying the QG condition, the sample complexity of SAA improves from $\cO(1/\epsilon^3)$ to $\cO(1/\epsilon^2)$ when assuming smoothness. This is very different from the classical results on the SO and the MSP, where Lipschitz smoothness plays no essential role in the sample complexity ~\cite{kleywegt2002sample,shapiro2006complexity}.  Our results are built on the traditional large deviation theory and stability arguments, while leveraging several bias-variance decomposition techniques, in order to fully exploit the specific structure of CSO and other structural assumptions.

\item[(b)] We analyze the sample complexity of the modified SAA in (\ref{eq:iSAA}) for the special case (\ref{eq:imain}),  where $\xi$ and $\eta$ are independent. We show that the total sample complexity of the modified SAA is  $\cO(d/\eps^2)$ for the general Lipschitz continuous problems. The existence of the QG condition only improves the complexity of the outer samples from $\cO(d/\eps^2)$ to $\cO(1/\eps)$, yet the overall complexity is dominated by the complexity of the inner samples, which is $\cO(d/\eps^2)$. Our complexity result  matches with the asymptotic rate established in ~\cite{dentcheva2017statistical} even without assuming smoothness of outer and inner functions and is unimprovable.

\item[(c)] We conduct some simulations of the SAA approach on several examples, including the logistic regression, least absolute value (LAV) regression and its smoothed counterpart, under some modifications.  Our simulation results indicate that solving the nonsmooth LAV regression requires more samples than solving its smooth counterpart to achieve the same accuracy. We also observe that when the variance of the inner randomness is relatively large, for a fixed budget $T$, setting $n=O(\sqrt{T})$ samples seems to perform best for logistic regression, which  matches with our theory. Although both conditional sampling and independent sampling schemes can be applied to solving the  special case (\ref{eq:imain}), with nearly matching sample complexity in situation (iv) (see last row in Table~\ref{tab:result}), our simulations show that using the independent sampling scheme exhibits better performance in practice. 
\end{enumerate}

\begin{table}[t]
	{\footnotesize
	\caption{Sample Complexity of SAA Methods}\label{tab:result}
	\renewcommand\arraystretch{1.5}
	\begin{center}
	\begin{tabular}{c|c|c|c|c}
		\toprule[1.5pt]
		
		\multirow{2}{*}{Problem}
		& \multicolumn{2}{c|}{Assumptions}
		&\multicolumn{2}{c}{Sample Complexity}
		\\
		\cline{2-5}
		
		{}
		& \emph{$f_\xi(\cdot)$}
		& \emph{$\hat F_n$ or $\hat F_{nm}$} 
		& \emph{Conditional}
		& \emph{Independent}
		\\
		\hline
		
	    {SO~\cite{kleywegt2002sample}}
		& \emph{-}
		& \emph{-}
		& $ \cO(d/\eps^2) $
		& \emph{-}
		\\
		\hline
		
		{SO~\cite{shalev2010learnability}}
		& \emph{-}
		& \emph{Strongly Convex}
		& $ \cO(1/\eps) $
		& \emph{-}
		\\
		\hline
		
		{MSP ($T=3$)~\cite{shapiro2006complexity}}
		& \emph{-}
		& \emph{-}
		& $ \cO(d^2/\eps^4) $
		& $ \cO(d^2/\eps^4) $
		\\
		\hline	
		
		{CSO}
		& \emph{-}
		& \emph{-}
		& $ \cO(d/\eps^4) $
		& \multirow{2}{*}{$ \cO(d/{\eps^2}) $}
		\\
		\cline{1-4} 
		
		{CSO}
		& \emph{Smooth}
		& \emph{-}
		& $ \cO(d/\eps^3) $
		& {}
		\\
		\hline	
		
		{CSO}
		& \emph{-}
		& \emph{Quadratic Growth}
		& $ \cO(1/\eps^3) $
		& \multirow{2}{*}{$ \cO(d/{\eps^2})$}
		\\
		\cline{1-4}	
		
		{CSO}
		& \emph{Smooth}
		& \emph{Quadratic Growth}
		& $ \cO(1/\eps^2) $
		& 
		\\
		\hline
		
		\multicolumn{5}{c}{$\hat F_n$ or $\hat F_{nm}$ = empirical objective; $\eps$ = accuracy; $d$ = dimension;} \\ 
		\multicolumn{5}{c}{ Conditional = conditional sampling; Independent = independent sampling;}\\ 
		\bottomrule[1.5pt]
	\end{tabular}
	\end{center}
}
\end{table}

\subsection{Paper Organization}
The remainder of this paper is organized as follows. In Section \ref{sec:pre}, we introduce some notations and preliminaries. In Section \ref{sec:problem}, we give the basic assumptions
and analyze the mean squared error of the Monte Carlo estimation. In Section \ref{sec:sca}, we present the main results on the sample complexity of SAA for CSO under different structural assumptions. In Section \ref{sec:indep}, we provide results for the special case when $\xi$ and $\eta$ are independent. Numerical results are given in Section \ref{sec:num}.

\section{Preliminaries}
\label{sec:pre}

For convenience, we collect here some notations that will be used throughout the paper. We also introduce some  mathematical tools and propositions that are necessary for future discussion. For simplicity, we restrict our attention to $l_2$-norm, denoted as $\norm{\cdot}_2$. Similar results on sample complexity with respect to different norms can be obtained with minor modification of the analysis.

Let $\mathcal{X} \subseteq \RR^d$ be the decision set. We say $\mathcal{X}$ has a finite diameter $D_\mathcal{X}$, if $\norm{x_1-x_2}_2 \leq D_\mathcal{X},$ $\forall x_1, x_2 \in \mathcal{X}$.  For $\upsilon \in (0,1)$,  $\{x_l\}_{l=1}^Q$ is said to be a $\upsilon$-net of $\mathcal{X}$, if $x_l \in \mathcal{X}$, $\forall l=1,\cdots,Q$, and the following holds: $\forall x \in \mathcal{X}, \exists l(x) \in \{1,\cdots,Q\} \text{ such that } ||x-x_{l(x)}||_2 \leq \upsilon.$ If $\mathcal{X}$ has a finite diameter $D_\mathcal{X}$, for any $\upsilon \in (0,1)$, there exists a $\upsilon$-net of $\mathcal{X}$, and the size of the $\upsilon$-net  is bounded,  $Q\leq \cO((D_\mathcal{X}/\upsilon)^d)$~\cite{shapiro2009lectures}.

A function $f: \mathcal{X} \rightarrow \RR$ is said to be $L$-Lipschitz continuous, if there exists a constant $L>0$ such that $|f(x_1) - f(x_2)|\leq L\norm{x_1-x_2}_2, \forall x_1, x_2 \in \mathcal{X}$. The function $f: \mathcal{X} \rightarrow \RR$ is said to be $S$-Lipschitz smooth, if it is continuously differentiable and its gradient is $S$-Lipschitz continuous. This also implies that  $\forall x_1, x_2 \in \mathcal{X}:|f(x_1)-f(x_2) -\nabla f(x_2)^\top(x_1-x_2)|  \leq  \frac{S}{2}\norm{x_1-x_2}_2^2$. If a continuously differentiable function $f: \mathcal{X} \rightarrow \RR$ satisfies that $\forall x_1, x_2 \in \mathcal{X}$, $f(x_1)-f(x_2)- \nabla f(x_2)^\top(x_1-x_2) \geq \frac{\mu}{2}\norm{x_1-x_2}_2^2,$ 
then $f$ is called $\mu$-strongly convex when $\mu>0$, convex when $\mu=0$, and $\mu$-weakly convex when $\mu<0$.

\begin{defn}[H\"olderian error bound  condition]
\label{def:lineq}
Let $f: \mathcal{X} \rightarrow \RR$ be a function with compact domain $\mathcal{X}$ and the optimal solution set $\mathcal{X}^*$ is nonempty. $f(\cdot)$ satisfies the $(\mu,\delta)$-H\"olderian error bound condition if there exists $\delta \geq 0$ and $\mu>0$ such that 
\begin{equation*}
\forall x\in \mathcal{X}, \quad f(x)-\min_{x\in\mathcal{X}}f(x) \geq \mu \inf_{z\in\mathcal{X}^*}\norm{x-z}_2^{1+\delta}. 
\end{equation*}
In particular, when $\delta=1$, we say $f$ satisfies the quadratic growth (QG) condition. 
\end{defn}

The H\"olderian error bound  condition is also known as the \L ojasiewicz inequality~\cite{Bolte2017}. When $\delta=1$, the condition implies a quadratic growth of the function value near any local minima. The QG condition is a weaker assumption than strong convexity and does not need to be convex. When $f(\cdot)$ is convex, the QG condition is also referred as optimal strong convexity~\cite{liu2015asynchronous} and semi-strong convexity~\cite{gong2014}. 

The Cram\'er's large deviation theorem will be frequently used, so we list it as a lemma below based on the result in~\cite{kleywegt2002sample}. We extend the result to random vectors and provide the proof in Appendix Section \ref{app:pre}.

\begin{lm}\label{lem:ldb}
Let $X_1,\cdots,X_n$ be i.i.d samples of zero mean random variable $X$ with finite variance $\sigma^2$. For any $\eps>0$, it holds
\begin{equation*}
    \PP\bigg(\frac{1}{n}\sum_{i=1}^n X_i \geq \eps\bigg) \leq \exp(-n I(\eps)),
\end{equation*}
where $I(\eps):=\sup_{t\in\RR} \{t\eps - \log M(t)\}$ is the rate function of random variable $X$, and $M(t):=\EE e^{tX}$ is the moment generating function of $X$. 
For any $\delta>0$, there exists $\eps_1>0$, for any $\eps \in (0,\eps_1)$,
$
I(\eps) \geq \frac{\eps^2}{(2+\delta)\sigma^2}.
$
If $X$ is a zero-mean sub-Gaussian, then
$
   \PP(\frac{1}{n}\sum_{i=1}^n X_i \geq \eps) \leq \exp(- \frac{n\eps^2}{2\sigma^2})$, $\forall \eps>0.
$

If $X$ is a zero-mean random vector in $\RR^k$ such that $\EE\|X\|_2^2= \sigma^2<\infty$,  then for any $\delta>0$, there exists $\eps_1>0$, for any $\eps \in (0,\eps_1)$,
\begin{equation*}
    \PP\bigg(\bigg\|\frac{1}{n}\sum_{i=1}^n X_i\bigg\|_2 \geq \eps\bigg) 
        \leq 2k\exp{\bigg(-\frac{n\eps^2}{(2+\delta)\sigma^2}\bigg)}.
\end{equation*}

\end{lm}

We will also use the simple fact that for any random variables $Y$ and $Z$, if random variable $W \leq X:=Y+Z$, then for any $\eps>0$, $\PP(W>\eps) \leq \PP(X>\eps) \leq \PP(Y>\frac{\eps}{2}) + \PP(Z>\frac{\eps}{2})$. Lastly, throughout the paper, we call $x_\eps\in\mathcal{X}$ an $\eps$-optimal solution to the problem $\min_{x\in\mathcal{X}} F(x)$, if $F(x_\eps)-\min_{x\in\mathcal{X}} F(x) \leq \eps$.

\section{Mean Squared Error of SAA Estimator for CSO}
\label{sec:problem}
In this section, we make the basic assumptions and analyze the mean squared error of the Monte Carlo estimate of the function value $f(x)$ at a given point. 
\subsection{Problem Formulation and Assumptions}
Recall the problem~(\ref{pro:ori}):
\begin{equation*}
\min_{x \in \mathcal{X}}\; F(x):=\EE_{\xi}\Big[f_\xi\Big({\EE_{\eta|\xi}[g_\eta(x,\xi)]}\Big)\Big],
\end{equation*}
where $f_\xi(\cdot):\RR^k\to\RR$, $g_\eta(\cdot,\xi):\RR^d\to\RR^k$ are random functions. Recall   its SAA counterpart~(\ref{pro:cosaa0}):
\begin{equation*}
\min_{x \in \mathcal{X}}\; \hat F_{nm}(x):= \frac{1}{n}\sum_{i=1}^n f_{\xi_i}\bigg(\frac{1}{m}\sum_{j=1}^m g_{\eta_{ij}}(x, \xi_{i})\bigg).
\end{equation*}
We denote $x^*$ and $ \hat{x}_{nm}$ the optimal solutions to the CSO and the SAA problems, respectively. We are interested in estimating the probability of $\hat x_{nm}$ being an $\eps$-optimal solution to the CSO problem, namely $\PP\left(F(\hat x_{nm})-F(x^*) \leq \eps\right)$, for an arbitrary accuracy $\eps>0$.

Throughout the paper, we assume availability of i.i.d. samples generated from distribution $\PP(\xi)$ and conditional distribution $\PP(\eta\given \xi)$ for any given $\xi$, and we make the following basic assumptions:
\begin{assume}
\label{as:est}
We assume that 
\begin{enumerate}
    \item[(a)] The decision set $\mathcal{X} \subseteq \RR^d$ has a finite diameter $D_\mathcal{X}>0$.
    \item[(b)] $f_\xi(\cdot)$ is $L_f$-Lipschitz continuous and $g_\eta(\cdot,\xi)$ is $L_g$-Lipschitz continuous for any given $\xi$ and $\eta$.
    \item[(c)] For all $x \in\mathcal{X}$, $f(x,\xi)$ is Borel measurable in $\xi$ and $g_\eta(x,\xi)$ is Borel measurable in $\eta$ for all $\xi$.
    \item[(d)] $\sigma_f^2:=\max_{x\in\mathcal{X}} \Var_\xi\left(f_\xi(\EE_{\eta|\xi}[g_\eta(x,\xi)])\right) < \infty$.
    \item[(e)] $\sigma_g^2:=\max_{x\in\mathcal{X}, \xi} \EE_{\eta|\xi}\norm{g_\eta(x,\xi)-{\EE_{\eta|\xi}} g_\eta(x,\xi)}_2^2 < \infty$.
    \item[(f)] $|f_\xi(\cdot)| \leq M_f$, $\|g_\eta(\cdot,\xi)\|_2\leq M_g$ for any $\xi$ and $\eta$.
\end{enumerate}
\end{assume}
The assumption (f) on  the boundedness of function values are implied from assumptions (a) and (b). The assumptions (d) and (e) on boundedness of variances are commonly used for sample complexity analysis in the literature.{The assumptions (b) and (c) together suggests that the function $f_\xi$ and $g_\eta(x,\xi)$ are Carath\'eodory functions~\cite{E_Kubi_ska_2005}.} Although the parameters $L_f$, $L_g$, $\sigma_f$, and $\sigma_g$ could depend on dimensions $d$ and $k$, we treat these parameters as given constants throughout the paper.

\subsection{Mean Squared Error of SAA Objective}
\label{sec:estimation}

In this subsection, we analyze the mean squared error (MSE) of the estimator $\hat F_{nm}(x)$, i.e., the SAA objective (or the empirical objective), for estimating the true objective function $F(x)$, at a given $x$.  The MSE can be decomposed into the sum of squared bias and variance of the estimator:
\begin{equation}
\label{eq:mse}
    \text{MSE} (\hat F_{nm}(x)) := \EE|\hat F_{nm}(x) - F(x)|^2 =(\EE\hat F_{nm}(x) - F(x))^2 + \Var(\hat F_{nm}(x)).  
\end{equation}

We have the following lemmas on bounding the bias and variance.
\begin{lm}\label{lm:bias} 
Let $\{\eta_{j}\}_{j=1}^m$ be conditional samples from $P(\eta|\xi)$ given $\xi\sim P(\xi)$. Under Assumption \ref{as:est}, for any fixed $x\in\cX$ that is independent of $\xi$ and $\{\eta_{j}\}_{j=1}^m$, it holds that,
\begin{equation}
\label{eq:lipschitz_bias}
\bigg|\EE_{\{\xi,\{\eta_{j}\}_{j=1}^m \}}\bigg[f_{\xi} \bigg(\frac{1}{m}\sum_{j=1}^m g_{\eta_{j}}(x,\xi)\bigg)   -   f_{\xi}\big(\EE_{\eta|\xi}g_{\eta}(x,\xi)\big)\bigg]\bigg| \leq \frac{ L_f \sigma_{g}}{\sqrt{m}}.
\end{equation}
If additionally, $f_\xi(\cdot)$ is $S$-Lipschitz smooth, we have
\begin{equation}
\label{eq:smooth_bias}
\bigg|\EE_{\{\xi,\{\eta_{j}\}_{j=1}^m\}} \bigg[f_{\xi} \bigg(\frac{1}{m}\sum_{j=1}^m g_{\eta_{j}}(x,\xi)\bigg)   -   f_{\xi}\big(\EE_{\eta|\xi}g_{\eta}(x,\xi)\big)\bigg]\bigg|  \leq  \frac{S\sigma_{g}^2}{2m}.
\end{equation}
\end{lm}
\begin{proof}{}
Define $X_{j}: = g_{\eta_{j}}(x,\xi)-\EE_{\eta|\xi}g_{\eta}(x,\xi)$ and $\bar{X}:=\sum_{j=1}^m X_j/m$. It follows  $\EE_{\{\eta_{j}\}_{j=1}^m|\xi}[\bar{X}]=0$ by definition, and $\EE_{\{\eta_{j}\}_{j=1}^m|\xi}[\|\bar{X}\|_2^2]\leq \sigma_g^2/m$ by Assumption \ref{as:est}(d).  $\EE_{\{\eta_{j}\}_{j=1}^m|\xi}  \nabla f_{\xi} \big(\EE_{\eta|\xi}g_{\eta}(x,\xi)\big)^\top \big(\frac{1}{m}\sum_{j=1}^m X_{j}(x)\big)=0$ since $x$ is independent of $\{\eta_j\}_{j=1}^m$. 
The results then follow directly by invoking the Lipschitz continuity and smoothness and taking expectations. 
\end{proof}

\begin{lm}\label{lm:variance}
Under Assumption \ref{as:est}, it holds that  $\Var(\hat F_{nm}(x)) \leq \frac{\sigma_{f}^2}{n}+  \frac{4M_f L_f\sigma_{g}}{n\sqrt{m}}.$
\end{lm}
\begin{proof}
We first introduce 
$\hat F_{n} (x):=\frac{1}{n}\sum_{i=1}^n f_{\xi_i}\big(\EE_{\eta|\xi_i}[g_{\eta}(x,\xi_i)]\big)$. It follows from the independence among $\{\xi_i\}_{i=1}^n$ that  $\Var(\hat F_{n}(x)) \leq \frac{\sigma_{f}^2}{n}$. By definition we have
\begin{equation*}
\begin{aligned}
&\Var\Big(\hat F_{nm}(x)\Big) - \Var\Big(\hat F_{n}(x)\Big)\\
    =&\frac{1}{n} \sbr{\EE(\hat F_{1m}(x)^2 ) - (\EE\hat F_{1m}(x))^2}-\frac{1}{n}\sbr{ (\EE(\hat F_{1} (x)^2)- (\EE\hat F_{1}(x))^2 }\\
    =&\frac{1}{n} \sbr{\EE(\hat F_{1m} (x)^2) -\EE(\hat F_{1} (x)^2)}+\frac{1}{n}\sbr{ (\EE\hat F_{1}(x))^2- (\EE\hat F_{1m}(x))^2},\\
\end{aligned}
\end{equation*}
where $\hat{F}_{1m}(x):=f_{\xi_1} \big(\frac{1}{m}\sum_{j=1}^m g_{\eta_{1j}}(x,\xi_1)\big)$ and $\hat{F}_{1}(x):=f_{\xi_1}\big(\EE_{\eta|\xi_1}g_{\eta}(x,\xi_1)\big)$.  From Assumption \ref{as:est}(b) and Lemma~\ref{lm:bias}, we have $\EE (\hat F_{1m}(x)^2) - \EE(\hat F_{1}(x)^2)\leq 2M_f\EE|\hat F_{1m}(x) -\hat F_{1}(x)|\leq 2M_fL_f\sigma_g/\sqrt{m}$. In addition, $(\EE\hat F_{1}(x))^2-(\EE\hat F_{1m}(x))^2
\leq  2M_fL_f\sigma_g/\sqrt{m}$. Hence, we obtain the desired result.   
\end{proof}
The following result on the mean squared error follows naturally by (\ref{eq:mse}).  
\begin{thm}
\label{thm:nested}
Under Assumption \ref{as:est}, we have
\begin{equation}
    \text{MSE} (\hat F_{nm}(x)) \leq \frac{L_f^2 \sigma_{g}^2}{m} + \frac{1}{n}\bigg(\sigma_f^2+\frac{4M_f L_f\sigma_g}{\sqrt{m}}\bigg).
\end{equation}
If additionally, $f_\xi(\cdot)$ is $S$-Lipschitz smooth,
the mean squared error is further bounded by 
\begin{equation}
    \text{MSE} (\hat F_{nm}(x)) \leq \frac{S^2\sigma_{g}^4}{4m^2} + \frac{1}{n}\bigg(\sigma_f^2+\frac{4M_f L_f\sigma_g}{\sqrt{m}}\bigg).
\end{equation}
\end{thm}

 Unlike the classical stochastic optimization, the SAA objective of CSO is no longer unbiased. The estimation error of the SAA objective therefore comes from both bias and variance. A key observation from Theorem~\ref{thm:nested} is that Lipschitz smoothness of $f_\xi(\cdot)$ is essential to reduce the bias and can be potentially exploited to improve the sample complexity of SAA. 

 We point out that in \cite{hong2009estimating}, the authors also consider the estimation problem of the expected value of  a non-linear function on a conditional expectation, i.e.,  $\EE[f(\EE[\zeta|\xi])]$. Their setting is slightly different from ours as they restrict $f$ to be one-dimensional and assume $f$ contains a finite number of discontinuous  or non-differential points and is thrice differentiable with finite derivatives on all continuous points. They provide an asymptotic  bound $\cO(1/m^2+1/n)$ of the mean squared error for their nested estimator based on Taylor expansion. Here we focus on a general continuous outer function $f_\xi(\cdot)$, and show that Lipschitz smoothness of $f_\xi(\cdot)$ is sufficient to achieve a similar error bound with finite samples.

\section{Sample Complexity of SAA for Conditional Stochastic Optimization}
\label{sec:sca}
In this section, we analyze the number of samples required for the solution to the SAA~(\ref{pro:cosaa0}) to be $\eps$-optimal of the CSO problem~(\ref{pro:ori}), with high probability.

We consider two general cases: (i) when the objective is Lipschitz continuous  and (ii) when the empirical objective satisfies the H\"olderian error bound  condition. In the former case, we establish a uniform convergence analysis based on concentration inequalities to bound $\PP(F(\hat{x}_{nm}) - F(x^*)\geq \eps)$,  and in the latter case, we provide a stability analysis. In both cases, we further take into account two scenarios, with and without the Lipschitz smoothness assumption of the outer function $f_\xi(\cdot)$. 

\subsection{Sample Complexity for General Lipschitz Continuous Functions}

We first consider the case when the objective is Lipschitz continuous and prove the uniform convergence.  

\begin{thm}[Uniform Convergence]
\label{thm:saa}
Under Assumption \ref{as:est}, for any $\delta >0$, there exists $\eps_1>0$ such that for $\eps \in (0,\eps_1)$, when 
$m \geq L_f^2 \sigma_{g}^2/\eps^2$, we have
\begin{equation}
\label{eq:ld}
\begin{split}
    \PP\bigg(\! \sup_{x\in \mathcal{X}}|\hat F_{nm}(x)\!-\!F(x) |\!> \!\eps\!\bigg) \!
    \leq\!
    \cO(1)\!\bigg(\!\frac{4L_fL_gD_\mathcal{X}}{\eps}\!\bigg)^d\!\exp\bigg(\!-\!\frac{n\eps^2}{16(2+\delta)(\sigma_f^2 + 4 M_f L_f\sigma_g)}\bigg).
\end{split}
\end{equation}
If additionally, $f_\xi(\cdot)$ is $S$-Lipschitz smooth, then  (\ref{eq:ld}) holds as long as  $m \geq 2S\sigma_{g}^2/\eps$. 
\end{thm}

\begin{proof}
We construct  a $\upsilon$-net to get rid of the supreme over $x$ and use a concentration inequality to bound the probability. First, we pick a $\upsilon$-net $\{x_l\}_{l=1}^{Q}$ on the decision set $\mathcal{X}$, such that $L_f L_g \upsilon = \eps/4$, thus  ${Q} \leq \cO(1){(\frac{4L_g L_f D_\mathcal{X}}{\eps})}^{d}$. Note that $\{x_l\}_{l=1}^{Q}$ has no randomness. By definition of $\upsilon$-net, we have 
$\forall x\in\mathcal{X}$, $\exists$  $l(x) \in \{1,2,\cdots,Q\}$, $\text{ s.t. }  \norm{x-x_{l(x)}}_2 \leq \upsilon = \eps/4 L_f L_g.$
Invoking Lipschitz continuity of $f_\xi$ and $g_\eta$, we obtain
\begin{equation*}
|\hat F_{nm}(x) -\hat F_{nm}(x_{l(x)})|\leq \frac{\eps}{4}, \quad |F(x) -F(x_{l(x)})| \leq \frac{\eps}{4}.
\end{equation*}
Hence, for any $x\in\cX$, 
\begin{equation*}
\begin{aligned}
    &|\hat F_{nm}(x)- F(x)| \\
    \leq ~&|\hat F_{nm}(x)-\hat F_{nm}(x_{l(x)})|+|\hat F_{nm}(x_{l(x)})-F(x_{l(x)})|+|F(x_{l(x)})- F(x)|\\
     \leq ~&\frac{\eps}{2}+|\hat F_{nm}(x_{l(x)})-F(x_{l(x)})|
     \leq \frac{\eps}{2}+\max_{l\in\{1,2,\cdots,Q\}} |\hat F_{nm}(x_{l})-F(x_{l})|.
\end{aligned}
\end{equation*}
It follows that 
\begin{equation}
\begin{aligned}
\label{smc:sep}
    \PP\bigg(\sup_{x\in \mathcal{X} } |\hat F_{nm}(x)- F(x)| > \eps\bigg)
    \leq & \PP\bigg(\max_{l\in\{1,2,\cdots,Q\}}|\hat F_{nm}(x_l)-F(x_l)| > \frac{\eps}{2}\bigg)\\
    \leq &\sum_{l=1}^{Q} \PP\bigg(|\hat F_{nm}(x_l)-F(x_l)| > \frac{\eps}{2}\bigg).
    \end{aligned}
\end{equation}
Define $Z_i(l):=f_{\xi_i}(\frac{1}{m}\sum_{j=1}^mg_{\eta_{ij}}(x_l,\xi_i))-F(x_l)$, then $Z_1(l), Z_2(l),\cdots, Z_n(l)$ are i.i.d. random variables. Denote their expectation as $\EE Z(l)$. Then $Z_i(l)-\EE Z(l)$ is a zero-mean random variable.

If $\max_l \EE Z(l) \leq \eps/4$, by Lemma \ref{lem:ldb}, we have 
\begin{equation}
\label{smc:po}
    \begin{aligned}
        &\PP\bigg(\hat F_{nm}(x_l)-F(x_l)> \frac{\eps}{2}\bigg)
         \leq \PP\bigg(\hat F_{nm}(x_l)-F(x_l) > \frac{\eps}{4}+\EE Z(l)\bigg)\\
        = & \PP\bigg(\frac{1}{n}\sum_{i=1}^n [Z_i(l)-\EE Z(l)] > \frac{\eps}{4}\bigg)
        \leq \exp\bigg(-\frac{n\eps^2}{16(\delta+2)\Var (Z(l))}\bigg).
    \end{aligned}
\end{equation}
Similarly, we could show that if $\max_l \EE Z(l) \geq -\eps/4$,
\begin{equation}
\label{smc:ne}
    \PP\bigg(F(x_l)-\hat F_{nm}(x_l)> \frac{\eps}{2}\bigg) \leq  \exp\bigg(-\frac{n\eps^2}{16(\delta+2)\Var (Z(l))}\bigg).
\end{equation}
Based on Lemma \ref{lm:bias}, we have, for Lipschitz continuous $f_\xi(\cdot)$, $|\EE Z(l)| \leq L_f\sigma_{g}/\sqrt{m}$, $ \forall l=1,\cdots,Q$; for Lipschitz smooth $f_\xi(\cdot)$,
$|\EE Z(l)| \leq S\sigma_{g}^2/2m$, $\forall l=1,\cdots,Q$. Thus, $\max_l\EE Z(l)\leq \eps/4$ is satisfied when $m$ is sufficiently large.  By analysis of Theorem \ref{thm:nested}, we know $\Var (Z(l)) \leq \sigma_f^2+4 M_f L_f \sigma_g /\sqrt{m} \leq \sigma_f^2+4 M_f L_f \sigma_g$.
Plugging into (\ref{smc:sep}) with ${Q} \leq \cO(1){(\frac{4L_g L_f D_\mathcal{X}}{\eps})}^{d}$, we obtain the desired result.
\end{proof}

Since $\hat F_{nm}(\hat x_{nm}) - \hat F_{nm}(x^*) \leq 0$, we have
\begin{equation}
\label{eq:decomposition}
\begin{split}
	&\PP\rbr{F(\hat x_{nm}) - F(x^*)\geq \eps}\\
	=\ &
	\PP\rbr{[F(\hat x_{nm})- \hat F_{nm}(\hat x_{nm})] + [\hat F_{nm}(\hat x_{nm}) - \hat F_{nm}(x^*)]+ [\hat F_{nm}(x^*) - F(x^*)]\geq \eps}\\
	\leq\ &
	\PP\rbr{F(\hat x_{nm})- \hat F_{nm}(\hat x_{nm})\geq \eps/2} + \PP\rbr{\hat F_{nm}(x^*) - F(x^*)\geq \eps/2}.
\end{split}
\end{equation}
Invoking Theorem \ref{thm:saa}, we immediately have the following result.

\begin{cor}[SAA under General Lipschitz Continuous Condition]
\label{cor:saa}
Under Assumption \ref{as:est}, for any $\delta >0$, there exists $\eps_1>0$ such that for $\eps \in (0,\eps_1)$, when 
$m \geq L_f^2 \sigma_{g}^2/\eps^2$,
\begin{equation}
\label{eq:saa}
    \PP\bigg( F(\hat x_{nm})-F(x^*) > \eps\bigg) 
    \leq 
    \cO(1)\bigg(\frac{8L_fL_gD_\mathcal{X}}{\eps}\bigg)^d\exp\bigg(\!-\frac{n\eps^2}{64(2+\delta)(\sigma_f^2 + 4 M_f L_f\sigma_g)}\bigg).
\end{equation}
If additionally, $f_\xi(\cdot)$ is $S$-Lipschitz smooth, then  (\ref{eq:saa}) holds as long as  $m \geq 2S\sigma_{g}^2/\eps$. 
\end{cor}
It further implies the following sample complexity result.  

\begin{cor} With probability at least $1-\alpha$, the solution to the SAA problem is $\epsilon$-optimal to the original CSO problem if the sample sizes $n$ and $m$ satisfy that
\begin{equation*}
\begin{aligned}
n &\geq \cO(1)\frac{\sigma_f^2 + 4 M_f L_f\sigma_g}{\eps^2}\bigg[d\log\rbr{\frac{8L_f L_g D_\mathcal{X}}{\eps}}+\log\rbr{\frac{1}{\alpha}}\bigg], \; \\
m &\geq
\begin{cases}
\frac{L_f^2 \sigma_{g}^2}{\eps^2}, &\text{ Under Assumption \ref{as:est}},\\
\frac{2S\sigma_{g}^2}{\eps},&\text{$f_\xi(\cdot)$ is also Lipschitz smooth.}
\end{cases}
\end{aligned}
\end{equation*}
Ignoring the log factors, under Assumption \ref{as:est}, the total sample complexity of SAA for achieving an $\epsilon$-optimal solution is  $T=mn+n = \cO(d/\eps^4);$
when $f_\xi(\cdot)$ is Lipschitz smooth, the total sample complexity reduces to 
$T=mn +n= \cO(d/\eps^3).$
\end{cor}

The above result indicates that in general, the sample complexity of the SAA for the CSO problem is $\cO(d/\eps^4)$ when assuming only Lipschitz continuity of the functions $f_\xi$ and $g_\eta$. The sample complexity drops to  $\cO(d/\eps^3)$ assuming additionally Lipschitz smoothness of the outer function $f_\xi$. Notice that the complexity depends only linearly on the dimension of the decision set. This is quite different from the three-stage stochastic optimization. In ~\cite{shapiro2006complexity}, for a three-stage stochastic programming, the authors showed the sample sizes for estimating the second and the third stages need to be at least $\cO(d/\eps^2)$, leading to a total of $\cO(d^2/\eps^4)$ samples, to guarantee uniform convergence even for stage-wise independent random variables.

\subsection{Sample Complexity under Error Bound Conditions}
In this subsection, we consider the case when the empirical function satisfies H\"olderian error bound condition, which includes the quadratic growth condition and strong convexity as special cases. Error bound condition has been widely studied recently in the context of (stochastic) oracle-based algorithm for faster convergence; see e.g., ~\cite{karimi2016linear,drusvyatskiy2018error,xu2016accelerated} and references therein. To our best knowledge, very few papers have exploited the H\"olderian error bound condition for the SAA approach and analyzed the sample complexity under such a condition. We show that the CSO problem under the H\"olderian error bound condition yields smaller orders of sample complexity for the SAA approach. 
We make the following two assumptions throughout this subsection. 
\begin{assume}
\label{as:fpl}
The empirical function $\hat F_{nm}(x)$ satisfies the $(\mu,\delta)$-H\"olderian error bound condition with $\mu>0,\delta\geq0$, i.e., it holds that 
\begin{equation*}
\forall x\in \mathcal{X}, \; \hat F_{nm}(x)-\min_{x\in\mathcal{X}}\hat F_{nm}(x) \geq \mu \inf_{z\in\mathcal{X}_{nm}^*}\norm{x-z}_2^{1+\delta}, 
\end{equation*}
where $n,m$ are any positive integers, and $\cX_{nm}^*$ is the optimal solution set of the empirical objective function $\hat F_{nm}(x)$ over $\cX$. 
\end{assume}

\begin{assume}
\label{as:uniquepro}
The empirical function $\hat F_{nm}$ has a unique minimizer $\hat x_{nm}$ on $\mathcal{X}$, for any $n$ and $m$.
\end{assume}

An interesting special case of Assumption~\ref{as:fpl} is the quadratic growth (QG) condition when $\delta=1$. QG condition is actually satisfied by a wide spectrum of objectives, such as strongly convex functions, general strongly convex functions composed with piecewise linear functions, general piecewise convex quadratic functions, etc. There are also many other specific examples arising in machine learning applications that satisfy the QG condition, including logistic loss composed with linear functions and neural networks with linear activation functions, see~\cite{pmlr-v80-charles18a, karimi2016linear}, and reference therein. Another interesting case is the polyhedral error bound condition when $\delta=0$, which is known to hold true for many piecewise linear loss functions~\cite{Bolte2017}. For both cases, these functions are not necessarily strongly convex nor convex. Relevant problems with SAA objective $\hat F_{nm}$ satisfying the QG condition are discussed in Appendix Section~\ref{app:qg}. 

Assumption~\ref{as:uniquepro} could be restricted and less straightforward to verify.  In general, for a non-strictly convex empirical objective function, the optimal solution is not necessarily unique. Yet,  it is not exclusive to strictly convex functions. We illustrate one such example below. Lastly, we point out that when $\hat F_{nm}(x)$ is strongly convex, for example, $l_2$ regularized convex empirical objective, the above assumptions hold naturally. In the following, we give some examples when $\hat F_{nm}(x)$ satisfies the QG condition.

\paragraph{Example 1}
Consider the following one-dimensional function 
\begin{equation*}
F(x)=\EE_{\xi}\big[(\EE_{\eta|\xi}[\eta]x)^2+3\sin^2(\EE_{\eta|\xi}[\eta]x)\big],
\end{equation*}
where $\xi$ and $\eta$ can be any random vectors that satisfy $\eta|\xi\geq\sqrt{\mu}$ with probability 1. Denote $\bar\eta_i=\frac{1}{m}\sum_{j=1}^m \eta_{ij}$, the empirical function is given by  
\begin{equation*}
\hat F_{nm} (x) = \frac{1}{n} \sum_{i=1}^n\bar\eta_i^2 x^2+\frac{3}{n} \sum_{i=1}^n\sin^2(\bar\eta_i x).
\end{equation*}
It can be easily verified that $\hat F_{nm} (x)$ satisfies the QG condition with parameter $\mu>0$. Moreover, the empirical function $\hat F_{nm} (x)$ has a unique minimizer $x^*=0$ for any $m,n$. 

\paragraph{Example 2} Consider the robust logistic regression problem with the objective 
\begin{equation}
\label{eq:logistic}
F(x)=\EE_{\xi=(a,b)}[\log(1+\exp(-b\EE_{\eta|\xi}[\eta]^Tx))],
\end{equation} 
where $a\in\RR^d$ is a random feature vector and $b\in\{1,-1\}$ is the label, $\eta=a+\cN(0,\sigma^2 I_d)$ is a perturbed noisy observation of the input feature vector $a$.  The empirical objective function $\hat F_{nm}(x)$ is given by
\begin{equation}
\label{eq:logisticSAA}
    \hat F_{nm} (x) = \frac{1}{n} \sum_{i=1}^n \log\bigg(1+\exp\bigg(-b_i \frac{1}{m}\sum_{j=1}^m \eta_{ij}^\top x\bigg)\bigg).
\end{equation}
$\hat F_{nm}(x)$ satisfies the QG condition on any compact convex set in Appendix Section~\ref{app:qg}. Note that the minimizer of a general empirical objective function is not necessarily always unique. However, the Hessian of $\hat F_{nm}(x)$ shows that $\hat F_{nm}(x)$ is strictly convex if $\frac{1}{m}\sum_{j=1}^m \eta_{ij}^\top\not = 0$ for all $i$, which is satisfied with high probability. Thus, $\hat F_{nm}(x)$ has a unique minimizer with high probability. 

Next, we present our main result on the sample complexity of SAA. 

\begin{thm}[SAA under Error Bound Condition]\label{thm:SAA_error_bound}
Under Assumption \ref{as:est}, \ref{as:fpl}, \ref{as:uniquepro}, for any $\eps > 0$, we have
\begin{equation}
\PP (F(\hat x_{nm}) - F(x^*) \geq \eps )
\leq 
\frac{1}{\eps}\bigg(L_f L_g\bigg(\frac{2L_f L_g}{\mu n}\bigg)^{1/\delta}
+
\frac{2L_{f}\sigma_{g}}{\sqrt{m}}\bigg).
\end{equation}
    
If additionally, $f_\xi(\cdot)$ is $S$-Lipschitz smooth, then we further have 
\begin{equation}
\PP (F(\hat x_{nm}) - F(x^*) \geq \eps )
\leq 
\frac{1}{\eps}\bigg(L_f L_g\bigg(\frac{2L_f L_g}{\mu n}\bigg)^{1/\delta}
+
\frac{S\sigma_{g}^2}{m}\bigg).
\end{equation}
\end{thm}

Different from the previous section, we use a stability argument to exploit the error bound condition. 
As shown in Lemma \ref{lm:bias}, the empirical function is a  biased estimator of the original function due to the composition of $f_\xi(\cdot)$ and $g_\eta(\cdot,\xi)$. Introducing a  perturbed set of samples could reduce some dependence in randomness. We define a bias term which will be used later in the proof:
\begin{equation}
\label{biasterm}
    \Delta(m):=
\begin{cases}
\frac{L_f\sigma_{g}}{\sqrt{m}}, &\text{ $f_\xi(\cdot)$ is $L_f$ Lipschitz continuous},\\
\frac{S\sigma_{g}^2}{2m},&\text{$f_\xi(\cdot)$ is additionally $S$ Lipschitz smooth.}
\end{cases}
\end{equation}
Below we provide the detailed proof of Theorem~\ref{thm:SAA_error_bound}. 

\begin{proof} 
Recall that $x^*$ and $\hat x_{nm}$ are the minimizers of $F(x)$ and $\hat F_{nm}(x)$, respectively. It's clear that $x^*$ has no randomness,  and $\hat x_{nm}$ is a function of $\{\xi_i\}_{i =1}^n, \{\eta_{ij}\}_{j=1}^m$. 
We decompose the error $F(\hat x_{nm}) - F(x^*)$ in three terms, and analyze each term below:
\begin{equation*}
\begin{split}
F(\hat x_{nm}) \!-\! F(x^*)
		\!=\!
		\underbrace{F(\hat x_{nm})\!-\! \hat F_{nm}(\hat x_{nm})}_{:=\Epsilon_1}+
		\underbrace{\hat F_{nm}(\hat x_{nm})\! -\! \hat F_{nm}(x^*)}_{:=\Epsilon_2}+
		\underbrace{\hat F_{nm}(x^*)\!- \! F(x^*)}_{:=\Epsilon_3}.
\end{split}
\end{equation*}
First, we use a stability argument and Lemma \ref{lm:bias} to bound $\EE \Epsilon_1 = \EE [F(\hat x_{nm})- \hat F_{nm}(\hat x_{nm})]$. 
Define 
\begin{equation}
\begin{aligned}
\label{eq:Fnmk}
     \hat F_{nm}^{(k)}(x) 
     := \frac{1}{n}\sum_{i\not = k}^n f_{\xi_i}\bigg(\frac{1}{m}\sum_{j =1}^m g_{\eta_{ij}}(x,\xi_i)\bigg) + \frac{1}{n}f_{\xi_k^{\prime}}\bigg(\frac{1}{m}\sum_{j =1}^m g_{\eta_{kj}^{\prime}}(x,\xi_k^{\prime})\bigg)
     \end{aligned}
\end{equation}
as the empirical function by replacing the $k^{th}$ outer sample $\xi_k$ with another i.i.d outer sample $\xi_k^{\prime}$, and replacing the corresponding  inner samples $\{\eta_{kj}\}_{j=1}^m$  with $\{\eta_{kj}^{\prime}\}_{j=1}^m$, which are sampled from the conditional distribution of $\PP(\eta|\xi_k^{\prime})$ for a given sample $\xi_k^\prime$. Denote $\hat x_{nm}^{(k)} := \argmin_{x \in \mathcal{X}} \hat F_{nm}^{(k)}(x)$. We decompose $\EE \Epsilon_1 = \EE [F(\hat x_{nm})- \hat F_{nm}(\hat x_{nm})]$ into three terms:
\begin{equation}\label{eq:decompose_term_1}
\begin{split}
\EE \Epsilon_1 = &\EE\bigg[\frac{1}{n}\sum_{k=1}^n F(\hat x_{nm}) -\frac{1}{n}\sum_{k =1}^n  f_{\xi_k} \bigg( \EE_{\eta|\xi_k} g_\eta(\hat x_{nm}^{(k)},\xi_k) \bigg)\bigg]\\
+&\EE \bigg[\frac{1}{n}\sum_{k =1}^n  f_{\xi_k} \bigg( \EE_{\eta|\xi_k} g_\eta(\hat x_{nm}^{(k)},\xi_k) \bigg) -\frac{1}{n}\sum_{k =1}^n f_{\xi_k}\bigg(\frac{1}{m}\sum_{j=1}^m g_{\eta_{kj}}(\hat x_{nm}^{(k)},\xi_k)\bigg)\bigg]\\
+&\EE \bigg[\frac{1}{n}\sum_{k =1}^n f_{\xi_k}\bigg(\frac{1}{m}\sum_{j=1}^m g_{\eta_{kj}}(\hat x_{nm}^{(k)},\xi_k)\bigg) - \hat F_{nm}(\hat x_{nm})\bigg].
\end{split}
\end{equation}
Note that $
    \EE [F(\hat x_{nm})] = \EE [F(\hat x_{nm}^{(k)})]
$ since $\xi_k$ and $\xi_k^{\prime}$ are i.i.d, which implies that  $\hat x_{nm}$ and $\hat x_{nm}^{(k)}$ follow an identical distribution.
Since $\hat x_{nm}^{(k)}$ is independent of $\xi_k$, 
$
\EE [F(\hat x_{nm}^{(k)})] = \EE [f_{\xi_k}(\EE_{\eta|\xi_k}g(\hat x_{nm}^{(k)},\xi_k))]
$
for any $k$. Then the first term in (\ref{eq:decompose_term_1}) is $0$.
As $\hat x_{nm}^{(k)}$ is independent of $\{\eta_{kj}\}_{j=1}^m$, the second term in (\ref{eq:decompose_term_1}) could be bounded by Lemma \ref{lm:bias},
it holds
\begin{equation}
\label{eq:biasboth}
    \EE \bigg[ f_{\xi_k} \bigg( \EE_{\eta|\xi_k} g_\eta(\hat x_{nm}^{(k)},\xi_k) \bigg)  
    -   
    f_{\xi_k}\bigg(\frac{1}{m}\sum_{j=1}^m g_{\eta_{kj}}(\hat x_{nm}^{(k)},\xi_k) \bigg)\bigg]
    \leq 
    \Delta(m).
\end{equation}
\if 0
for $S$-Lipschitz smooth $f_\xi(\cdot)$, we get
\begin{equation*}
    \EE \bigg[ f_{\xi_k} \bigg( \EE_{\eta|\xi_k} g_\eta(\hat x_{nm}^{(k)},\xi_k) \bigg)  
    -   
    f_{\xi_k}\bigg(\frac{1}{m}\sum_{j=1}^m g_{\eta_{kj}}(\hat x_{nm}^{(k)},\xi_k) \bigg)\bigg] 
    \leq 
    \frac{S\sigma_{g}^2}{2m}.
\end{equation*}
\fi
For the third term in (\ref{eq:decompose_term_1}), by definition it implies
\begin{equation}
\label{eq:sca1}
    \begin{aligned}
    \hat F_{nm}(\hat x_{nm}^{(k)})  - \hat F_{nm}(\hat x_{nm}) = &\hat F_{nm}^{(k)}(\hat x_{nm}^{(k)})  - \hat F_{nm}^{(k)}(\hat x_{nm}) \\
    + &
      \frac{1}{n}f_{\xi_k}\! \bigg(\! \frac{1}{m}\! \sum_{j =1}^m\!  g_{\eta_{kj}}(\hat x_{nm}^{(k)},\xi_k)\! \bigg) \! 
    - \! \frac{1}{n}f_{\xi_k}\! \bigg(\! \frac{1}{m}\sum_{j =1}^m g_{\eta_{kj}}(\hat x_{nm},\xi_k)\! \bigg)\\
     + &\frac{1}{n}f_{\xi_k^{\prime}}\! \bigg(\! \frac{1}{m}\sum_{j =1}^m g_{\eta_{kj}^{\prime}}(\hat x_{nm},\xi_k^{\prime})\! \bigg) \! 
   \!  - \! 
   \frac{1}{n}f_{\xi_k^{\prime}}\! \bigg(\! \frac{1}{m}\sum_{j =1}^m g_{\eta_{kj}^{\prime}}(\hat x_{nm}^{(k)},\xi_k^{\prime})\! \bigg).
    \end{aligned}
\end{equation}
By Lipschitz continuity of $f_\xi$ and $g_\eta$ and that $\hat F_{nm}^{(k)}(\hat x_{nm}^{(k)})  - \hat F_{nm}^{(k)}(\hat x_{nm}) \leq 0$, it holds
\begin{equation}
\label{eq:sca2}
\hat F_{nm}(\hat x_{nm}^{(k)})  - \hat F_{nm}(\hat x_{nm}) \leq \frac{2}{n}L_f L_g \norm{\hat x_{nm}^{(k)}-\hat x_{nm}}_2.
\end{equation}
Since $\hat x_{nm}$ is the unique minimizer of $\hat F_{nm}(x)$, and $\hat F_{nm}(x)$ satisfies QG condition with parameter $\mu$, we have
\begin{equation}
\label{eq:strong}
    \hat F_{nm}(\hat x_{nm}^{(k)})  - \hat F_{nm}(\hat x_{nm}) \geq \mu \norm{\hat x_{nm}^{(k)}-\hat x_{nm}}_2^{1+\delta}.
\end{equation}
Combining with (\ref{eq:sca2}), we obtain
\begin{equation}
\label{eq:strong2}
\norm{\hat x_{nm}^{(k)}-\hat x_{nm}}_2 \leq \bigg(\frac{2L_f L_g}{\mu n }\bigg)^{1/\delta}.
\end{equation}
By Lipschitz continuity of $f_\xi(\cdot)$ and $g_\eta(\cdot,\xi)$, and definition of $\hat F_{nm}(\hat x_{nm})$, we obtain
\begin{equation}
\label{eq:sca}
\begin{split}
   \EE \bigg[\frac{1}{n}\sum_{k =1}^n f_{\xi_k}\bigg(\frac{1}{m}\sum_{j=1}^m g_{\eta_{kj}}(\hat x_{nm}^{(k)},\xi_k)\bigg) - \hat F_{nm}(\hat x_{nm})\bigg]
    \leq  L_f L_g\bigg(\frac{2L_f L_g}{\mu n}\bigg)^{1/\delta}.
\end{split}
\end{equation}

Combining (\ref{eq:decompose_term_1}), (\ref{eq:sca}), and (\ref{eq:biasboth}), we obtain
\begin{equation}
\label{eq:scAl}
    \EE \Epsilon_1
    \leq 
    L_f L_g\bigg(\frac{2L_f L_g}{\mu n}\bigg)^{1/\delta}+\Delta(m).
\end{equation}
\if 0
for $L_f$-Lipschitz continuous $f_\xi(\cdot)$,
\begin{equation}
\label{eq:scAl}
    \EE \Epsilon_1
    \leq 
    L_f L_g\bigg(\frac{2L_f L_g}{\mu n}\bigg)^{1/\delta}+\frac{L_{f}\sigma_{g}}{\sqrt{m}};
\end{equation}
for $S$-Lipschitz smooth $f_\xi(\cdot)$,
\begin{equation}
\label{eq:scAs}
    \EE \Epsilon_1
    \leq 
    L_f L_g\bigg(\frac{2L_f L_g}{\mu n}\bigg)^{1/\delta}+\frac{S\sigma_{g}^2}{2m}.
\end{equation}
\fi

Second, by optimality of $\hat x_{nm}$ of $\hat F_{nm}$, we have
\begin{equation}
\label{eq:scB}
    \EE \Epsilon_2 = \EE [\hat F_{nm}(\hat x_{nm})- \hat F_{nm}(x^*)] \leq 0.
\end{equation}

Next, we bound $\EE \Epsilon_3$.
Define
$
\hat{F}_n(x):=\frac{1}{n}\sum_{i=1}^nf_{\xi_i}\big(\EE_{\eta|\xi_i}[g_\eta(x,\xi_i)]\big).
$
Notice that $x^*$ is independent of $\{\eta_{ij}\}_{j=1}^m$ for any $i = \{1,\cdots,n\}$ and $\EE [\hat F_n(x^*) - F(x^*)]=0$. By Lemma \ref{lm:bias}, it holds
\begin{equation}
\label{eq:scC}
    \EE \Epsilon_3 = \EE[ \hat F_{nm}(x^*)-\hat F_n(x) ]+ \EE[\hat F_n(x) - F(x)]
    \leq 
    \Delta(m);
\end{equation}
\if 0
for $L_f$-Lipschitz continuous $f_\xi(\cdot)$,
\begin{equation}
\label{eq:scCL}
    \EE \Epsilon_3 = \EE[ \hat F_{nm}(x^*)-\hat F_n(x) ]+ \EE[\hat F_n(x) - F(x)]
    \leq 
    \frac{L_{f}\sigma_{g}}{\sqrt{m}};
\end{equation}
for $S$-Lipschitz smooth $f_\xi(\cdot)$, we have
\begin{equation}
\label{eq:scCS}
    \EE \Epsilon_3 = \EE[ \hat F_{nm}(x^*)-\hat F_n(x) ]+ \EE[\hat F_n(x) - F(x)]
    \leq 
    \frac{S\sigma_{g}^2}{2m}.
\end{equation}
\fi
Combining (\ref{eq:scAl}), (\ref{eq:scB}), (\ref{eq:scC}), with Markov inequality, we obtain the desired result.
\end{proof} 
The sample complexity of SAA under the H\"olderian error bound condition follows directly.

\begin{cor}\label{cor:SAA_error_bound} Under Assumption \ref{as:fpl} and \ref{as:uniquepro}, with probability at least $1-\alpha$, the solution to the SAA problem is $\epsilon$-optimal to the original CSO problem if the sample sizes $n$ and $m$ satisfy that 
\begin{equation*}
n \geq \frac{(2L_fL_g)^{\delta+1}}{\mu(\alpha\eps)^\delta}, \;\; 
m\geq
\begin{cases}
\frac{16L_f^2 \sigma_{g}^2}{\alpha^2\eps^2}, &\text{ Under Assumption \ref{as:est}},\\
\frac{2S\sigma_{g}^2}{\alpha\eps},&\text{$f_\xi(\cdot)$ is also Lipschitz smooth.}
\end{cases}
\end{equation*}
Hence, the total sample complexity of SAA for achieving an $\epsilon$-optimal solution is at most  $T=mn+n = \cO(1/\eps^{\delta+2});$
when $f_\xi(\cdot)$ is Lipschitz smooth, the total sample complexity reduces to 
$T=mn +n= \cO(1/\eps^{\delta+1}).$
\end{cor}

In particular, when the empirical function is strongly convex or satisfies the QG condition, i.e., Assumption~\ref{as:fpl} with $\delta=1$, this leads to the total sample complexity of $\cO(1/\eps^{3})$ for Lipschitz continuous case and $\cO(1/\eps^{2})$ for Lipschitz smooth case, respectively. From the above corollary, the error bound condition only affects the sample complexity of the outer samples, and the sample size decreases as $\delta$ decreases.   As $\delta$ gets closer to zero, the sample complexity will essentially be dominated by the inner sample size. 

A key difference between the results in Theorems~\ref{thm:saa} and \ref{thm:SAA_error_bound} lies in the dependence on the problem dimension $d$ and confidence level $\alpha$. While the sample complexity under the H\"olderian error bound condition is dimension-free, the dependence on the confidence level $1-\alpha$ grows from $\cO(\log(1/\alpha))$ to $\cO(1/\alpha^\delta)$. This is similar to classical results on stochastic optimization for strongly convex objectives~\cite{shalev2010learnability}. Theorem \ref{thm:SAA_error_bound} could also be used to derive a dimensional free sample complexity of $l_2$ regularized SAA for a general convex CSO problem. See Appendix Section~\ref{app:regularied} for more details.

\section{Sample Complexity of SAA for CSO with Independent Random Variables}
\label{sec:indep}
In this section, we consider the special case of CSO when the random variables $\xi$ and $\eta$ are independent. The objective then simplifies to: 
\beq{eq:imain2}
    \min_{x\in \mathcal{X} }\quad {F}(x)
    := 
    \EE_\xi[f_\xi(\EE_{\eta}[g_{\eta}(x,\xi)])]. 
\eeq
This is similar yet slightly more general than~(\ref{eq:wang}), the compositional objective considered in \cite{wang2016accelerating, wang2017stochastic}. Note that the inner cost function we consider here is dependent on both $\xi$ and $\eta$, and thus cannot be written as a composition of two deterministic functions. 

The sample complexity of SAA under the conditional sampling setting achieved in Section~\ref{sec:sca} applies to this setting since it can be viewed as a special case of the former. However, since the inner expectation is no longer a conditional expectation, we now consider an alternative modified SAA, using the independent sampling scheme, in which we use the same set of samples to estimate the inner expectation. The procedure of the independent sampling scheme for solving~(\ref{eq:imain2}) works as follows: first generate $n$ i.i.d. samples $\{\xi_i\}_{i=1}^n$ from the distribution of $\xi$; and $m$ i.i.d samples $\{\eta_{j}\}_{j=1}^m$ from the distribution of $\eta$, then solve the following approximation problem:
\beq{eq:iSAA0}
\min_{x \in \mathcal{X}}\quad \hat F_{nm}(x):= \frac{1}{n}\sum_{i=1}^n f_{\xi_i}\bigg(\frac{1}{m}\sum_{j=1}^m g_{\eta_{j}}(x, \xi_{i})\bigg).
\eeq

As a result, the total sample complexity becomes $T=m+n$. 
In recent work by \cite{dentcheva2017statistical}, the authors established a central limit theorem result for the SAA~(\ref{eq:iSAA0}) with $m=n$. In particular, they have shown that for Lipschitz smooth functions $f_\xi(\cdot)$ and $g_\eta(\cdot,\xi) = g_\eta(\cdot)$,  the SAA estimator converges in distribution as follows:
\begin{equation*}
\sqrt{m}\left(\min_{x\in\mathcal{X}}\hat F_{mm}(x) - \min_{x\in\mathcal{X}} F(x)\right)\to Z(W)
\end{equation*}
where $W(\cdot)=(W_1(\cdot),W_2(\cdot))$ is a zero-mean Brownian process with certain covariance functions and $Z(\cdot)$ is a function that depends on the first order information. This result only yields an asymptotic convergence rate of order $\cO(1/\sqrt{m})$ for the SAA with $m=n$.
Below, we will provide a finite sample analysis for SAA and establish refined sample complexity results based on concentration inequality techniques. 

In the SAA problem (\ref{eq:iSAA0}), the component functions $f_{\xi_i}\big(\frac{1}{m}\sum_{j=1}^m g_{\eta_j}(x,\xi_i)\big)$ share the same random vectors $\{\eta_j\}_{j=1}^m$ and are dependent. This is distinct from the SAA (\ref{pro:cosaa0}) considered in the previous section. Because of this key difference, the previous analysis will no longer apply to this modified SAA. We will resort to a different analysis for deriving the sample complexity. Similarly, we consider two structural assumptions, when the empirical objective is only known to be Lipschitz continuous and when the empirical objective also satisfies  the error bound condition. 

\subsection{Sample Complexity for Lipschitz Continuous Problems}
We first consider the case when the objective is Lipschitz continuous. We make the
same basic assumptions of the Lipschitz continuity of $f_\xi(\cdot)$ and $g_\eta(\cdot,\xi)$ and boundedness of variances as described in Assumption~\ref{as:est}. Our main result is summarized below.  

\begin{thm}
\label{thm:indep_uniform_convergence}
Under the independent sampling scheme and Assumption \ref{as:est}, for any $\delta>0$, there exists an $\eps_1>0$ such that for any $\eps \in (0,\eps_1)$, it holds
\begin{equation}
\label{eq:ldI}
\begin{aligned}
    &\PP\bigg( \sup_{x\in \mathcal{X}}|\hat F_{nm}(x)-F(x) |> \eps\bigg) \\
     \leq &
    \cO(1)\bigg( \frac{4L_f L_g D_\mathcal{X}}{\eps} \bigg)^d \bigg(\exp\bigg(-\frac{n\eps^2}{16(\delta+2)\sigma_f^2}\bigg)
    +  nk\exp\bigg(-\frac{m\eps^2}{16(\delta+2) L_f^2 \sigma_g^2}\bigg)\bigg).
    \end{aligned}
\end{equation}
Here, $d$ is the dimension of the decision set, and $k$ is the dimension of the range of function $g$.
\end{thm}

\begin{proof}
First, we pick a $\upsilon$-net $\{x_l\}_{l=1}^{Q}$ on the decision set $\mathcal{X}$, such that $L_f L_g \upsilon = \eps/4$. Using a similar argument in the proof of Theorem \ref{thm:saa}, we obtain 
\begin{equation}
\label{smc:Isep}
\begin{aligned}
    & \PP\bigg(\sup_{x\in \mathcal{X} } |\hat F_{nm}(x)- F(x)| > \eps\bigg)
     \leq \sum_{l=1}^{Q} \PP\bigg(|\hat F_{nm}(x_l)-F(x_l)| > \frac{\eps}{2}\bigg) \\
     \leq &\sum_{l=1}^{Q} \PP\bigg(|\hat F_{nm}(x_l)-\hat F_n (x_l)| >\frac{\eps}{4}\bigg)+ \sum_{l=1}^{Q} \PP\bigg(|\hat F_{n}(x_l)-F(x_l)| > \frac{\eps}{4}\bigg).
    \end{aligned}
\end{equation}
By Lipschitz continuity of $f_\xi(x)$ and Lemma~\ref{lem:ldb}, we have
\begin{equation}
\label{eq:fnm_fn}
\begin{aligned}
\PP\bigg(|\hat F_{nm}(x_l)-\hat F_n (x_l)|\geq \frac{\eps}{4}\bigg)
\leq & 
\sum_{i=1}^n \PP \bigg(\norm{\frac{1}{m}\sum_{j=1}^m g_{\eta_{j}}(x_l,\xi_i)-\EE_{\eta}g_{\eta}(x_l,\xi_i)}_2 \geq \frac{\eps}{4L_f}\bigg) \\
\leq &
2nk\exp\bigg(-\frac{m\eps^2}{16(\delta+2)L_f^2 \sigma_g^2}\bigg).
\end{aligned}
\end{equation}
By Lemma \ref{lem:ldb}, we obtain
\begin{equation}
\begin{aligned}
\PP\bigg(|\hat F_{n}(x_l)-F (x_l)|\geq \frac{\eps}{4}\bigg) 
& \leq 2\exp\bigg(-\frac{n\eps^2}{16(\delta+2)\sigma_f^2}\bigg).
\end{aligned}
\end{equation}
Combining with the fact that ${Q} \leq \cO(1){(\frac{4L_g L_f D_\mathcal{X}}{\eps})}^{d}$, we obtain the desired result. 
\end{proof}
Invoking the relation in (\ref{eq:decomposition}), the above theorem  implies the following:
\begin{cor}
\label{cor:isaa}
Under Assumption \ref{as:est}, with probability at least $1-\alpha$, the solution to the modified SAA problem~(\ref{eq:iSAA0}) is $\epsilon$-optimal to the original problem~(\ref{eq:imain2}) if the sample sizes $n$ and $m$ satisfy 
\begin{eqnarray*}
n &\geq& \frac{\cO(1)\sigma_f^2}{\eps^2}\bigg[d\log\rbr{\frac{8L_f L_g D_\mathcal{X}}{\eps}}+\log\rbr{\frac{1}{\alpha}}\bigg], \\
m&\geq& \frac{\cO(1)L_f^2\sigma_g^2}{\eps^2}\bigg[d\log\rbr{\frac{8L_f L_g D_\mathcal{X}}{\eps}}+\log\rbr{\frac{1}{\alpha}}+\log\rbr{nk}\bigg]. 
\end{eqnarray*}
Ignoring the log factors, under Assumption \ref{as:est}, the total sample complexity of the modified SAA for achieving an $\epsilon$-optimal solution is  $T=m+n = \cO(d/\eps^2)$.
\end{cor}

Note that this sample complexity is significantly smaller than that for the general CSO. The $\cO(d/\eps^2)$ sample complexity also matches the lower bounds on sample complexity of SAA for classical stochastic optimization with Lipschitz continuous objectives~\cite{massart2006risk}; therefore, this result is  unimprovable without further assumptions.    

\subsection{Sample Complexity Under Error Bound Conditions}
We now consider the case when the empirical objective satisfies Assumption~\ref{as:fpl} and~\ref{as:uniquepro}, i.e., the empirical objective $\hat F_{nm}(x)$ satisfies the error bound condition and has a unique minimizer for any integers $n,m$. Our main result is summarized as follows.
\begin{thm} 
\label{thm:iscsaa}{}
Under Assumptions \ref{as:est}, \ref{as:fpl}, and \ref{as:uniquepro}, for any $\eps >0$ and $\upsilon>0$, we have
\begin{equation}
\label{eq:convergence_iscsaa}
\begin{aligned}
    &\PP (F(\hat x_{nm}) - F(x^*) \geq \epsilon )\\
    \leq &\frac{1}{\eps}\bigg(
    L_f L_g\bigg(\frac{2L_f L_g}{\mu n}\bigg)^{1/\delta}+
    \cO(1)\frac{L_fM_g\sqrt{d\log(D_\mathcal{X}/\upsilon)}}{\sqrt{m}} +\frac{L_f\sigma_g}{\sqrt{m}}
    + 2\upsilon L_f L_g\bigg).
    \end{aligned}
\end{equation}
The solution to the modified SAA problem~(\ref{eq:iSAA0}) is $\epsilon$-optimal to the  problem~(\ref{eq:imain2}) with probability at least $1-\alpha$, if 
$\upsilon = \frac{\eps \alpha}{12L_fL_g}$, and 
the sample sizes $n$ and $m$ satisfy that
\begin{equation}
\label{eq:isamplecomplexity}
    n \geq \frac{(2L_fL_g)^{\delta+1}}{\mu(\alpha\eps)^\delta},\;
    m \geq \max\left\{\left(\frac{12L_f\sigma_g}{\alpha\eps}\right)^2,\cO(1)\left(\frac{6L_fM_g}{\alpha\eps}\right)^2d\log\left(\frac{12D_\mathcal{X}L_fL_g}{\alpha\eps}\right) \right\}.
\end{equation}
\end{thm}

Similar to Theorem~\ref{thm:SAA_error_bound}, the outer sample size is independent of dimension and decreases as $\delta$ decreases. As $\delta$ gets closer to zero, the sample complexity will essentially be dominated by the inner sample size. In particular, when the empirical function satisfies the QG condition or is strongly convex, i.e., Assumption~\ref{as:fpl} holds with $\delta=1$, the outer sample size is reduced from $\cO(d/\eps^{2})$ in the Lipschitz continuous case to $\cO(1/\eps)$. Yet, the total sample complexity remains $\cO(d/\eps^{2})$. 

For a CSO problem with independent random vectors (\ref{eq:imain2}), both SAA approaches, through conditional sampling, or independent sampling, can be applied to solve the problem. Comparing Theorem~\ref{thm:SAA_error_bound} and Theorem~\ref{thm:iscsaa}, when smoothness and the quadratic growth condition are satisfied,  the sample complexities of these two SAA approaches achieve the same order $\cO(1/\epsilon^2)$, except for an extra $O(d)$ factor for the independent sampling. Interestingly, for a given small dimension $d$ and the same sample budget $T$, the independent sampling might outperform the conditional sampling scheme since the constant factor in the sample complexity of conditional sampling is much larger. The numerical experiment on our testing cases  in the next section further supports the finding. 

In contrast to the sample complexity established in Section~\ref{sec:sca} for the conditional sampling setting, a notable difference here is that the Lipschitz smoothness condition does not necessarily help reduce the sample complexity. This result  aligns with the central limit theorem established in~\cite{dentcheva2017statistical}. One of the reasons arises from the interdependence among the component functions in the modified SAA objective, leading to  extra variance. Because of that, the analysis  requires sophisticated arguments to handle the dependence and is much more involved . We defer the proof to Appendix Section~\ref{sec:appendix-ind}.

\begin{remark} Although the overall $\cO(1/\eps^2)$ sample complexity cannot be further improved in general, it is worth pointing out that, for some interesting specific instances, the modified SAA could achieve lower sample complexity than what is described from theory. We illustrate this from the following example. 
\end{remark}

\paragraph{Example 3}
For $\gamma >0$, consider the following problem
\begin{equation*}
\min_{x\in\mathcal{X}} F(x):=  H(\EE_\eta [x+\eta],\gamma)+(\EE_\eta [x+\eta])^2,
\end{equation*}
where $\eta\sim\cN(0,\sigma_\eta^2)$ and $H(\cdot,\gamma)$ is the Huber function, i.e., 
\begin{equation}
H(x,\gamma) = 
\left\{
\begin{aligned}
& |x|-\frac{1}{2}\gamma && \text{for }{|x|>\gamma}.\\
&\frac{1}{2\gamma}x^2 &&  \text{for }{|x|\leq\gamma}.
\end{aligned}
\right.
\end{equation}
Note that here $f_\xi(x) :=f(x)= H(x,\gamma)+x^2$ is deterministic, and $g_\eta(x,\xi) = x+\eta$. When $\gamma >0$, $f(x)$ is $1/\gamma$-Lipschitz smooth. When $\gamma\rightarrow 0$, $f(x) \rightarrow |x|+x^2$, which is no longer differentiable. In this  example, $x^* = \argmin_{x\in\mathcal{X}}F(x) = -\EE \eta$, $F^* = \min_{x\in\mathcal{X}}F(x) =  0$. The empirical objective becomes
$
\hat F_{m}(x) = H(x+\bar \eta, \gamma)+(x+\bar \eta)^2,
$
where $\bar \eta =  \frac{1}{m}\sum_{j=1}^m \eta_j$. Thus, $\hat x_{m} =\argmin_{x\in\mathcal{X}}\hat F_{m}(x) =  -\bar \eta$. We show that  the error of SAA satisfies
\begin{equation}
\label{example3}
0 \leq \EE F(\hat x_{m}) - F(x^*) - \bigg(\frac{\sigma_\eta^2}{2\gamma m} \text{erf}\bigg(\sqrt{\frac{\gamma^2 m}{2\sigma_\eta^2}}\bigg)+\frac{\sigma_\eta^2}{m}\bigg) \leq  \sqrt{\frac{\sigma_\eta^2}{2\pi m}}\exp\bigg(-\frac{m\gamma^2}{2\sigma_\eta^2}\bigg),
\end{equation}
where erf$(x) := \frac{2}{\sqrt{\pi}}\int_{0}^x \exp(-x^2)dx$. As a result, when $\gamma\rightarrow 0$,
\begin{equation}
\label{example3_1}
\lim_{\gamma\to 0}\EE F(\hat x_{m}) - F(x^*) = \sqrt{\frac{\sigma_\eta^2}{2\pi m}}+\frac{\sigma_\eta^2}{m}.
\end{equation}
For completeness, we provide detailed derivation in Appendix Section~\ref{sec:appendix-ex}. This example shows that the SAA error improves from $\cO(1/\sqrt{m})$ to $\cO(1/m)$ as the objective transits from nonsmooth to smooth. When $\gamma\to 0$, the function becomes non-Lipschitz differentiable and the $O(1/\sqrt{m})$ bound for this setting is indeed tight. It remains an interesting open problem to identify sufficient conditions for achieving theoretically better sample complexity under the independent sampling scheme. 

\section{Numerical Experiments}
\label{sec:num}

In this section, we conduct numerical experiments based on two applications, logistic regression and robust regression, to demonstrate the performance of SAA for solving CSO problems. For a fixed sample budget $T$, we adopt difference sample allocation strategies for $(m,n)$, and compute the corresponding accuracy of the SAA estimators. We repeat 30 runs for each sample allocation  and report the average performance.  The SAA problems are solved by CVXPY 1.0.9 ~\cite{cvxpy}.  
\subsection{Robust Logistic Regression}
We consider the robust logistic regression problem in Example 2.  The problem is formulated in (\ref{eq:logistic}) and its SAA counterpart is of the form (\ref{eq:logisticSAA}) with domain $\mathcal{X} = \{x|x\in\RR^d, \|x\|_2\leq 100\}$. 

Note that from Example 2, $f$ is Lipschitz-smooth, $\hat F_{nm}(x)$ satisfies QG condition on any compact convex set, and with high probability has a  unique minimizer for large $n$. Theorem \ref{thm:SAA_error_bound} implies that the theoretical optimal sample allocation strategy is $n = \cO(1/\sqrt{T})$ and $m =\cO(1/\sqrt{T})$.

In the experiment, we set $d = 10$ and the samples of $\xi =(a, b)$ and $\eta$ are generated as follows: $a_i \sim \cN(0,\sigma_\xi^2I_{d})$, $b_i = \{\pm 1\}$ according to the sign of $a_i^Tx^*$, $\eta_{ij}\sim \cN(a_i,\sigma_\eta^2I_{d})$. We set $\sigma_\xi^2 = 1$, and consider three cases for $\sigma_\eta$: $\sigma_\eta^2=\{0.1, 10, 100\}$, corresponding low, medium, high variances from inner randomness.  For a given sample budget $T$ ranging from $10^3$ to $10^6$, four different sample allocation strategies are considered, i.e. $n=[T^{1/4}]$, $n=[T^{1/3}]$, $n=[T^{1/2}]$, and $n=[T^{2/3}]$. We then compute the average estimation error $F(\hat{x}_{nm})-F^*$ over 30 runs and its standard deviation. The results are summarized in Figure~\ref{Fig:logistic}, where $x$-axis denotes the sample budget $T$, %ranging from $10^3$ to $10^6$, 
and $y$-axis shows the estimation error. Each curve represents a sampling scheme,
showing the average error and upper confidence bound.  

\begin{figure}[t]
\label{Fig:logistic}
\centering
\begin{minipage}{.32\textwidth}
\subfigure[$\sigma_\eta^2/\sigma_\xi^2=0.1$]
{\includegraphics[width=\textwidth]{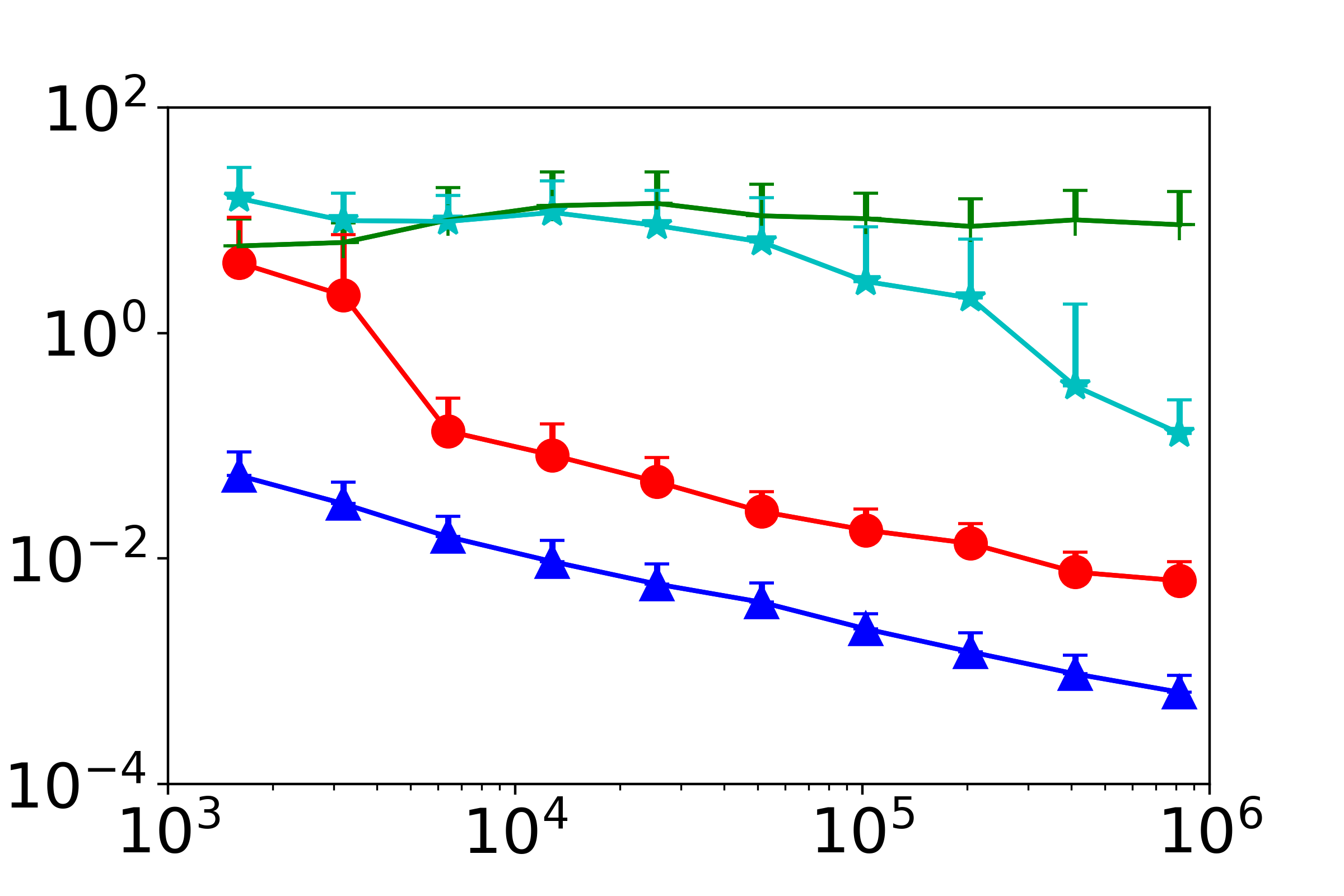}}
\end{minipage}%
\begin{minipage}{.32\textwidth}
\subfigure[$\sigma_\eta^2/\sigma_\xi^2=10$]
{\includegraphics[width=\textwidth]{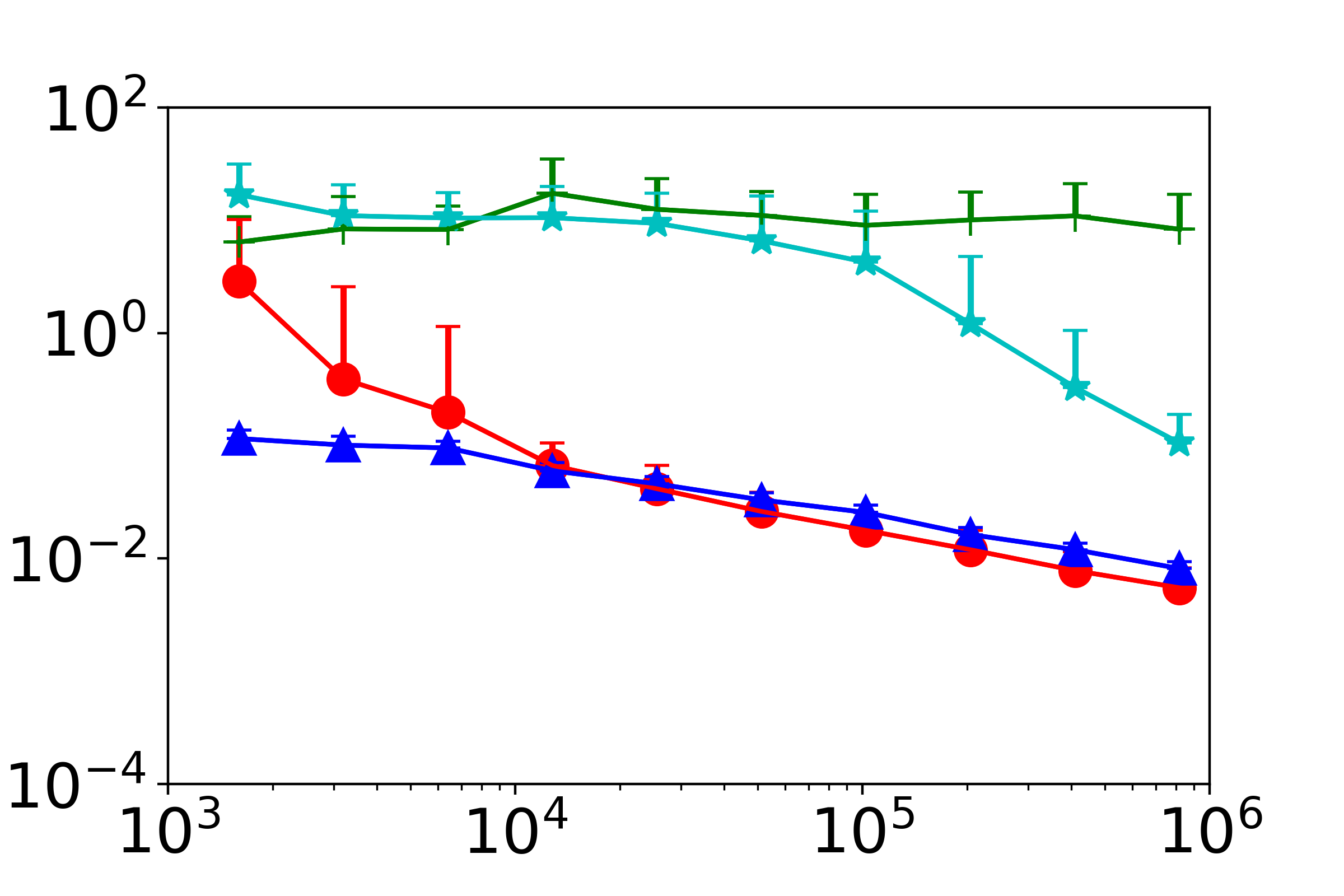}}
\end{minipage}
\begin{minipage}{.32\textwidth}
\subfigure[$\sigma_\eta^2/\sigma_\xi^2=100$]
{\includegraphics[width=\textwidth]{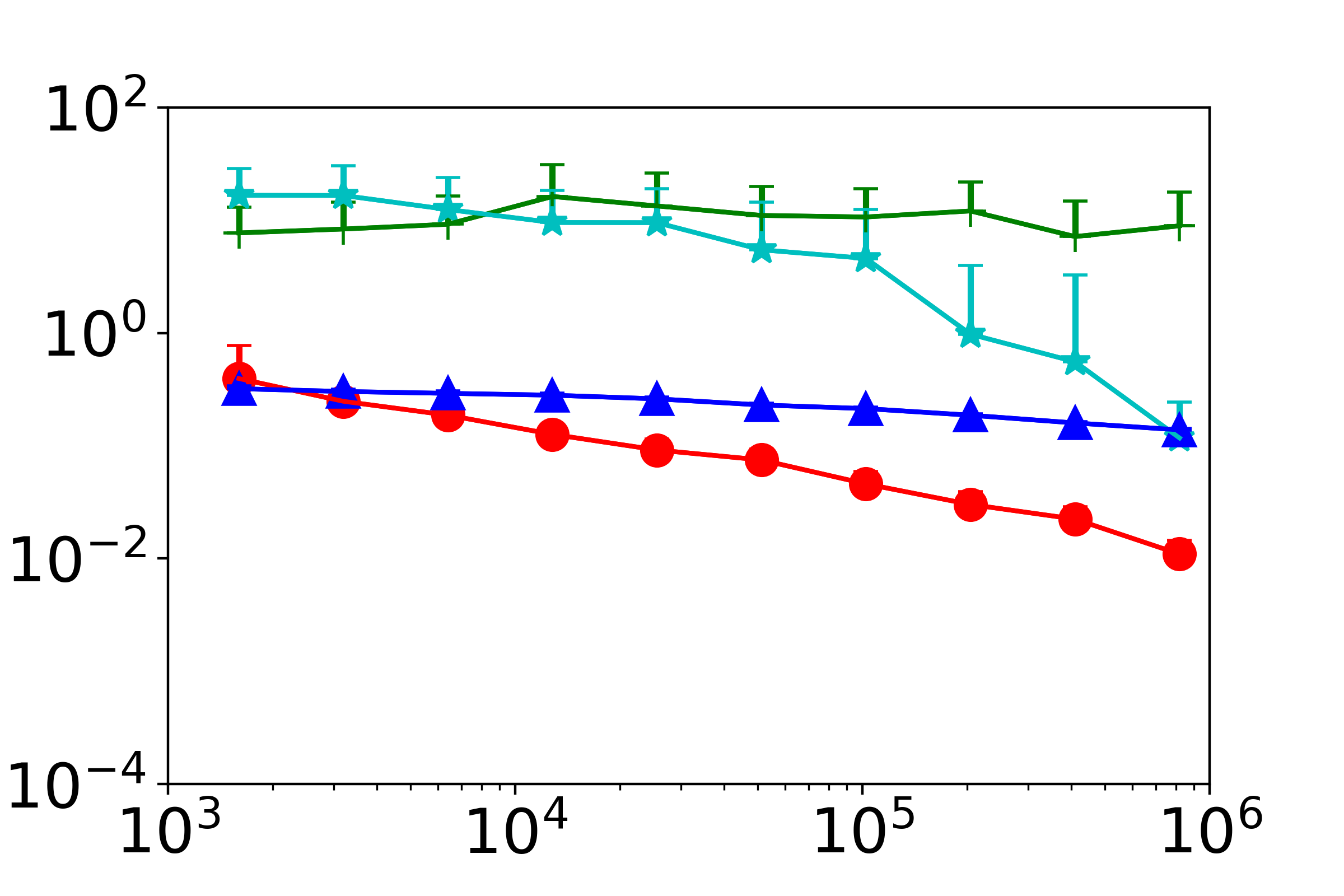}}
\end{minipage}%

\includegraphics[width=0.8\textwidth]{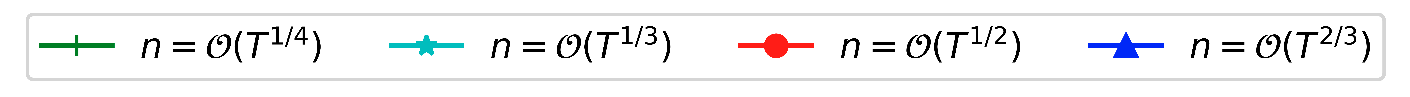}

\caption{Logistic regression, conditional sampling, dimension $d=10$}
\end{figure}
 
The trend from Figure~\ref{Fig:logistic}(a)-(c) shows that when the inner variance is relatively large, setting $n=\cO(T^{1/2})$ consistently outperforms the other sampling strategies, which matches our analysis. The error bar suggests that larger number of outer samples results in smaller deviation of the estimation accuracy.

\subsection{Robust Regression}
We now examine the robust regression problem, where the objective is no longer Lipschitz differentiable. The problem is as follows:
\begin{equation}
\min_{x\in\mathcal{X}}\; F(x) = \EE_{\xi=(a,b)}|\EE_{\eta|\xi} \eta^\top x - b|,
\end{equation}
where $a\in\RR^d$ is a random feature vector and $b\in \RR$ is the label, $\eta=a+\cN(0,\sigma_\eta^2 I_d)$ is a perturbed noisy observation of the input feature vector $a$, and the domain is $\mathcal{X} = \{x|x\in\RR^d, \|x\|_2\leq 100\}$.   For comparison purposes, we also consider the smoothed version of this problem based on the Huber function:
\begin{equation*}
\min_{x\in\mathcal{X}} \; F^\gamma(x)=\EE_{\xi=(a,b)}H\left(\EE_{\eta\given\xi}\eta^\top x-b ,\gamma \right),
\end{equation*}
where $\gamma>0$ is the smoothness parameter. 
The empirical functions for these two objectives are given by
\begin{equation*}
\hat F_{nm}(x) = \frac{1}{n}\sum_{i=1}^n \bigg|\frac{1}{m}\sum_{i=1}^m \eta_{ij}^\top x-b_i \bigg|, \quad
\hat F^\gamma_{nm}(x) = \frac{1}{n}\sum_{i=1}^n H\bigg(\frac{1}{m}\sum_{i=1}^m \eta_{ij}^\top x-b_i ,\gamma \bigg).
\end{equation*}
Theorem~\ref{thm:saa} and Theorem~\ref{thm:SAA_error_bound} indicate that Lipschitz smoothness of outer function $f_\xi(x)$ helps reduce the inner sample size required to achieve the same level of accuracy. For a given budget $T$, the theoretical optimal sample allocation strategies for these two problems  is $n = \cO(T^{1/2})$ and $n =\cO(T^{2/3})$, respectively.

In our experiment, we set $d=20$. Samples of $\xi=(a,b)$ and $\eta$ are generated as follows: $a_i \sim \cN(0,\sigma_\xi^2I_{d})$, $b_i = a_i^\top x^*$,  $\eta_{ij} \sim \cN(a_i,\sigma_\eta^2 I_{d})$. As in the previous experiment, we measure the average error and upper confidence bound for both problems  with sample budget $T$ ranging from $10^3$ to $10^6$ under four different sample allocation strategies over 30 runs. We also consider two sets of smoothness parameters,  $\gamma \in\{0.1, 10\}$. The results are summarized in  Figure~\ref{fig:dependent_absolute}. 

\begin{figure}[ht] 
\label{fig:dependent_absolute}
\centering
\begin{minipage}{.33\textwidth}
\subfigure[$\sigma_\eta^2/\sigma_\xi^2=0.1$, absolute value]{
\includegraphics[width=.95\textwidth]{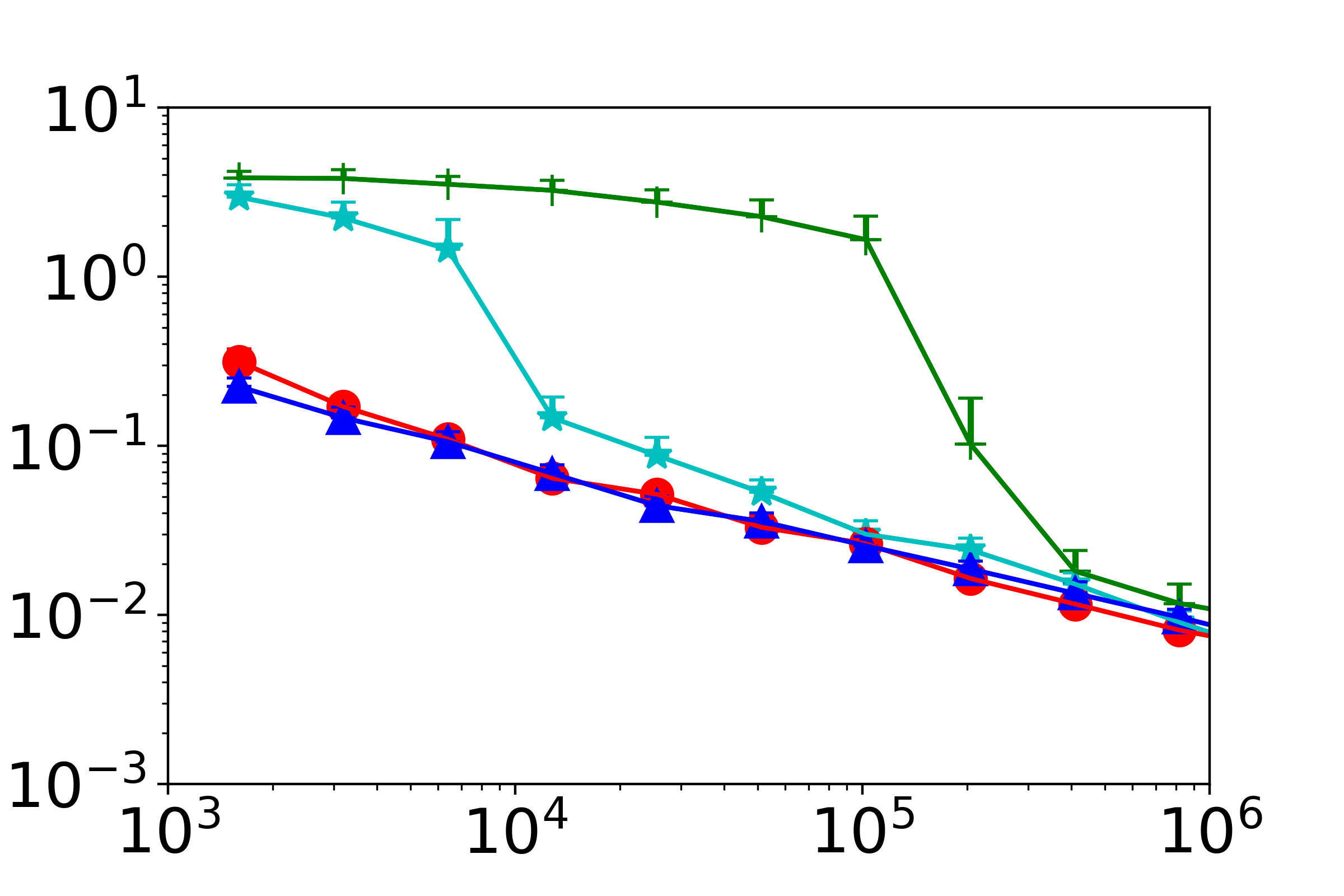}}
\end{minipage}%
\begin{minipage}{.33\textwidth}
\subfigure[$\sigma_\eta^2/\sigma_\xi^2=10$, absolute value]{
\includegraphics[width=.95\textwidth]{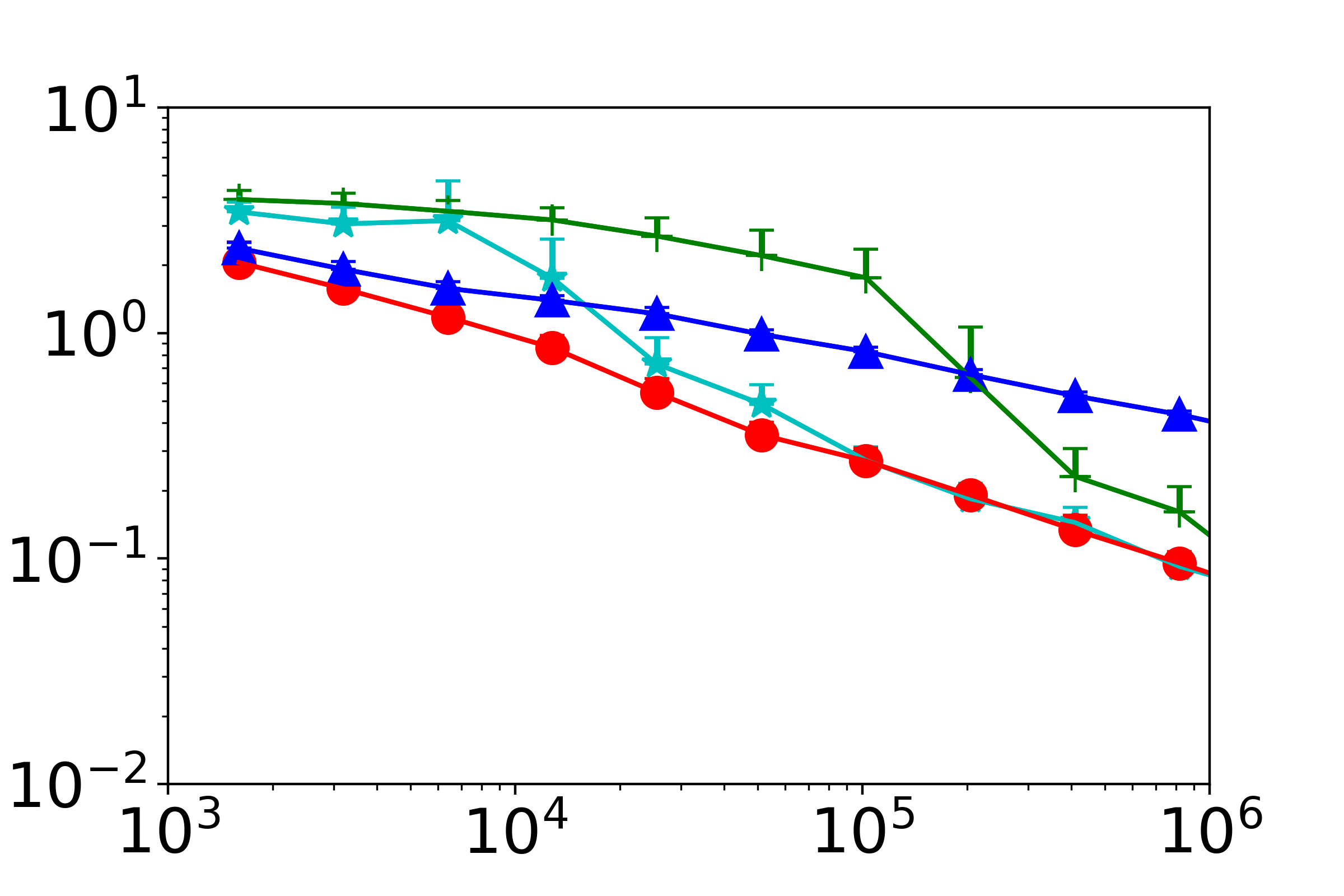}}
\end{minipage}
\begin{minipage}{.33\textwidth}
\subfigure[$\sigma_\eta^2/\sigma_\xi^2=100$, absolute value]{
\includegraphics[width=.95\textwidth]{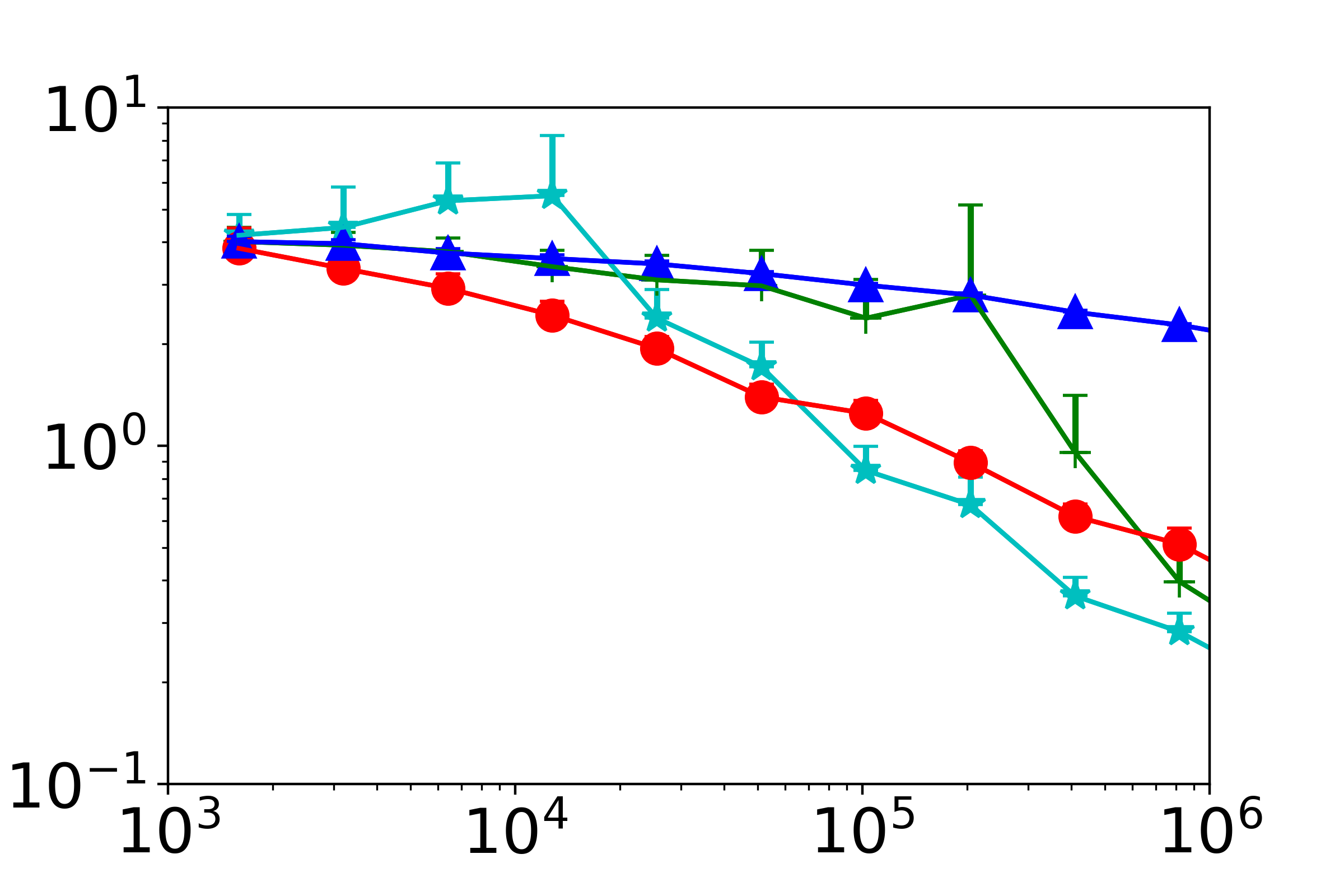}}
\end{minipage}%

\begin{minipage}{.33\textwidth}
\subfigure[$\sigma_\eta^2/\sigma_\xi^2=0.1$, Huber, $\gamma=0.1$]{
\includegraphics[width=.95\textwidth]{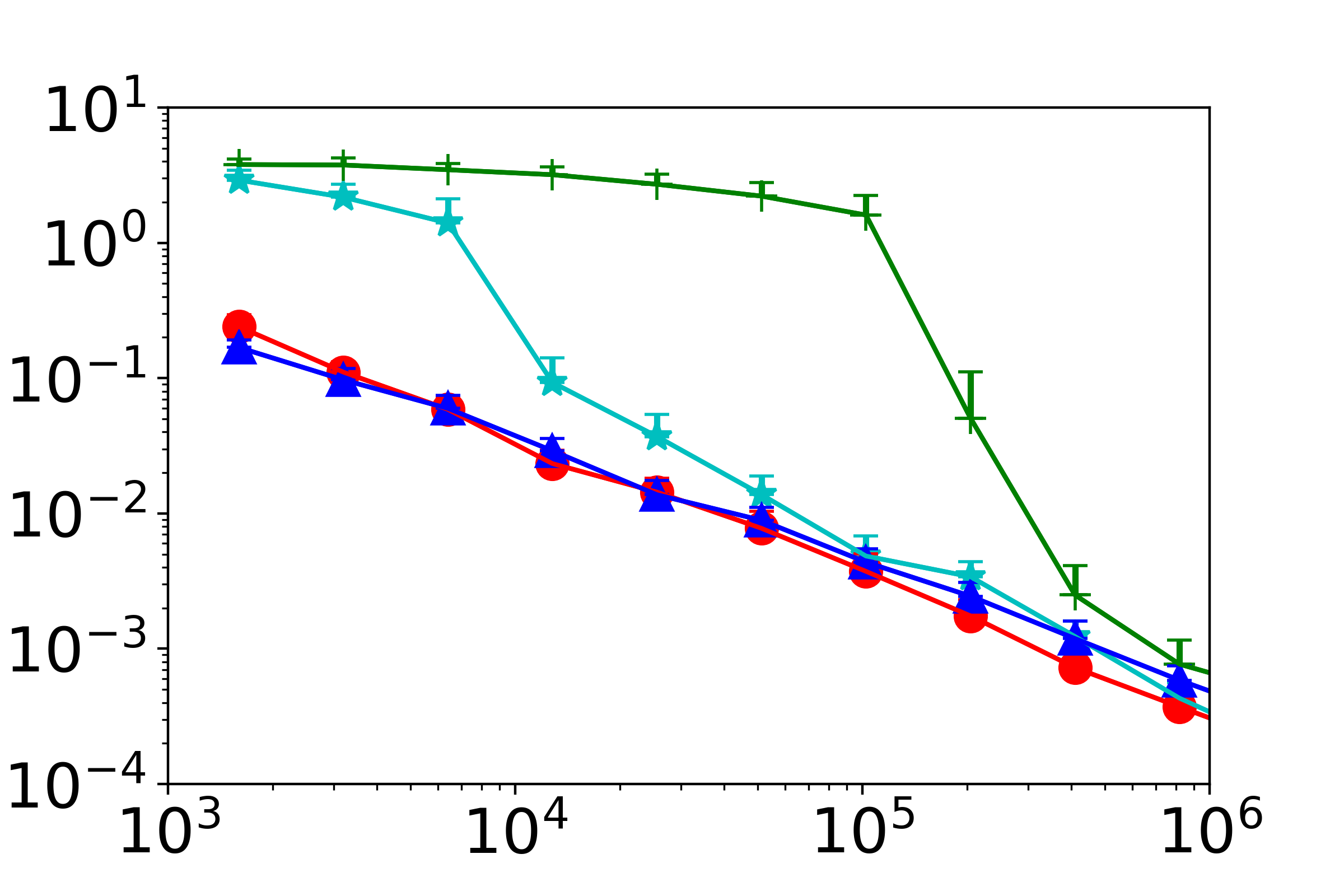}}
\end{minipage}%
\begin{minipage}{.33\textwidth}
\subfigure[$\sigma_\eta^2/\sigma_\xi^2=10$, Huber, $\gamma=0.1$]{
\includegraphics[width=.95\textwidth]{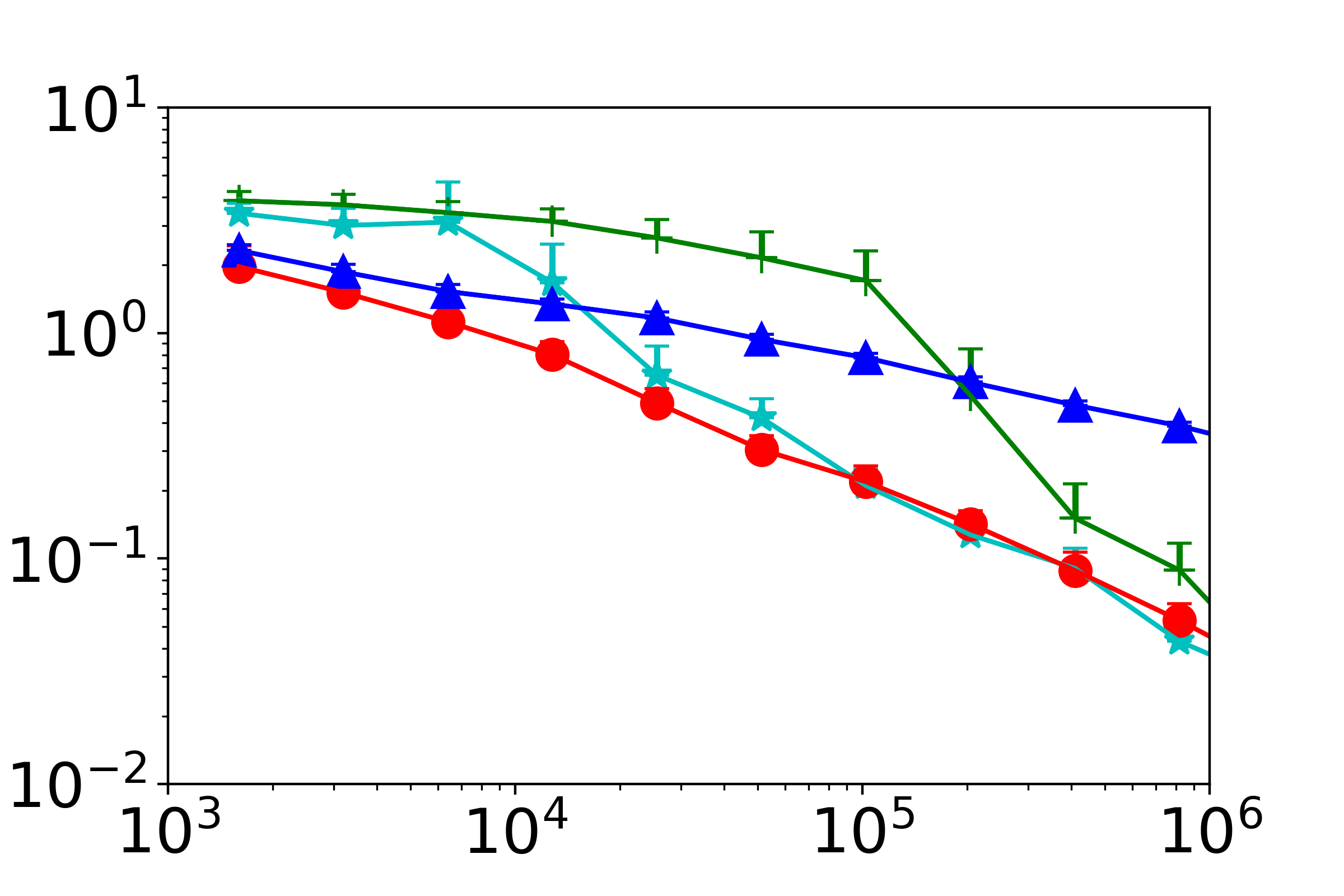}}
\end{minipage}
\begin{minipage}{.33\textwidth}
\subfigure[$\sigma_\eta^2/\sigma_\xi^2=100$, Huber, $\gamma=0.1$]{
\includegraphics[width=.95\textwidth]{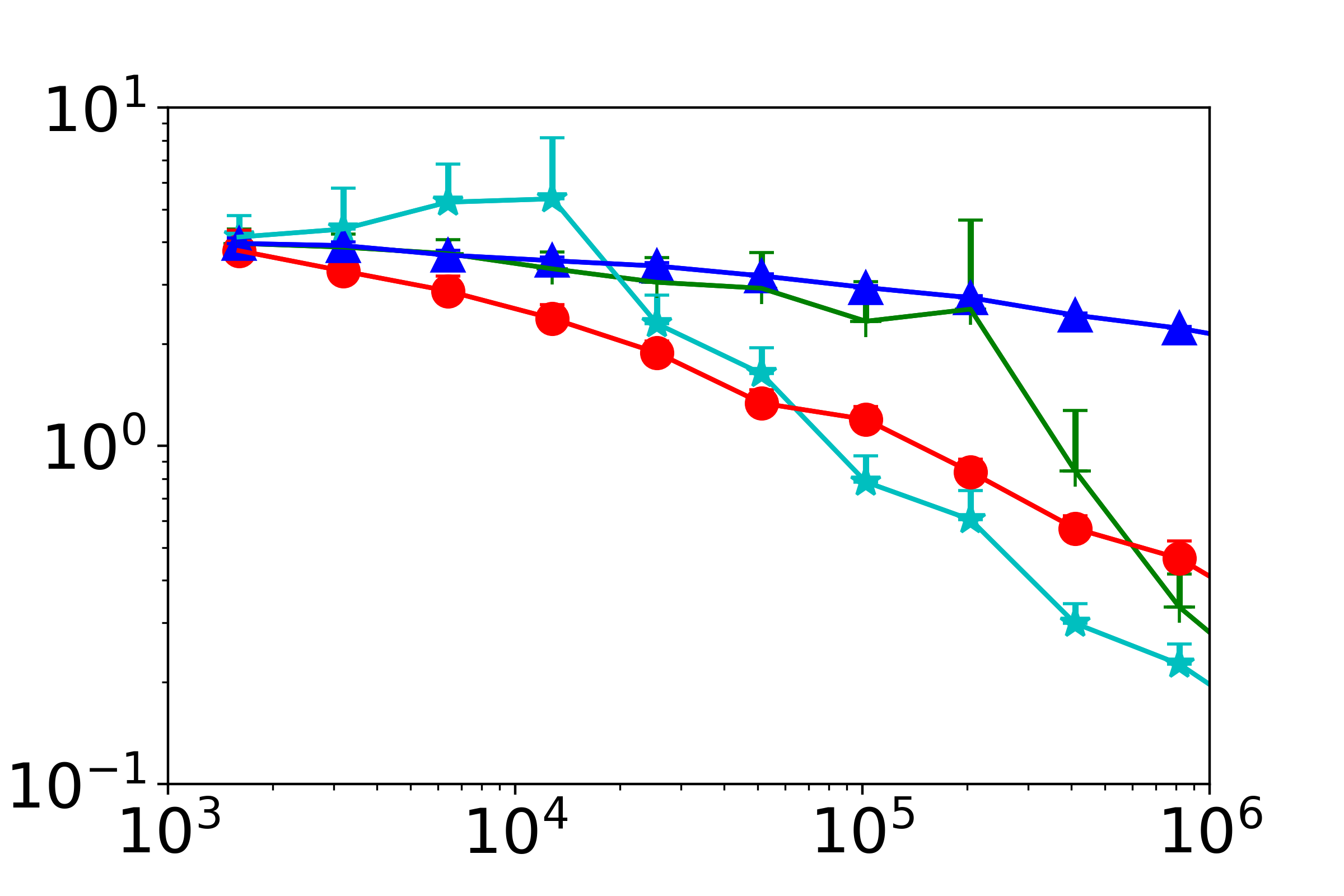}}
\end{minipage}%

\begin{minipage}{.33\textwidth}
\subfigure[$\sigma_\eta^2/\sigma_\xi^2=0.1$, Huber, $\gamma=10$]{
\includegraphics[width=.95\textwidth]{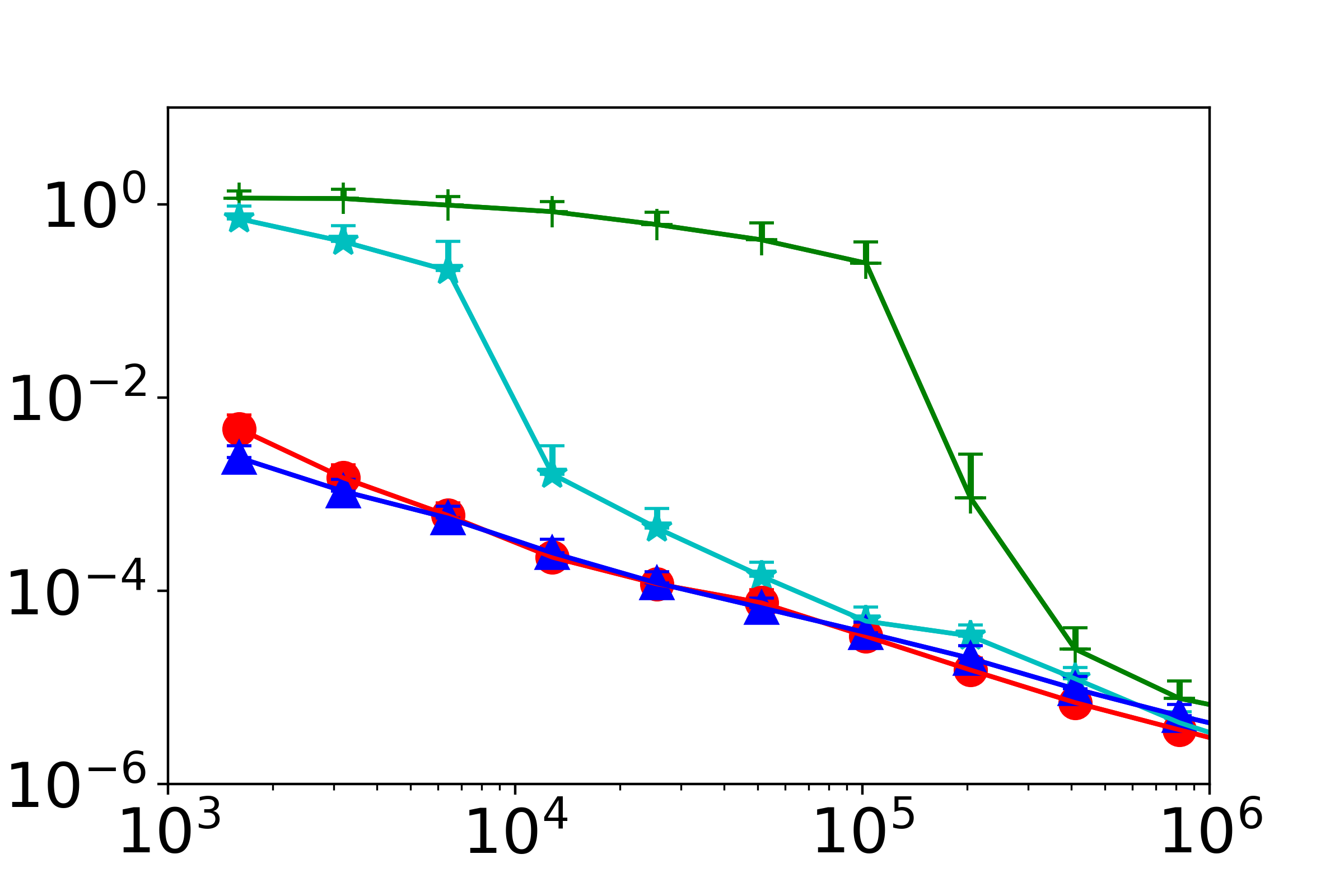}}
\end{minipage}%
\begin{minipage}{.33\textwidth}
\subfigure[$\sigma_\eta^2/\sigma_\xi^2=10$, Huber, $\gamma=10$]{
\includegraphics[width=.95\textwidth]{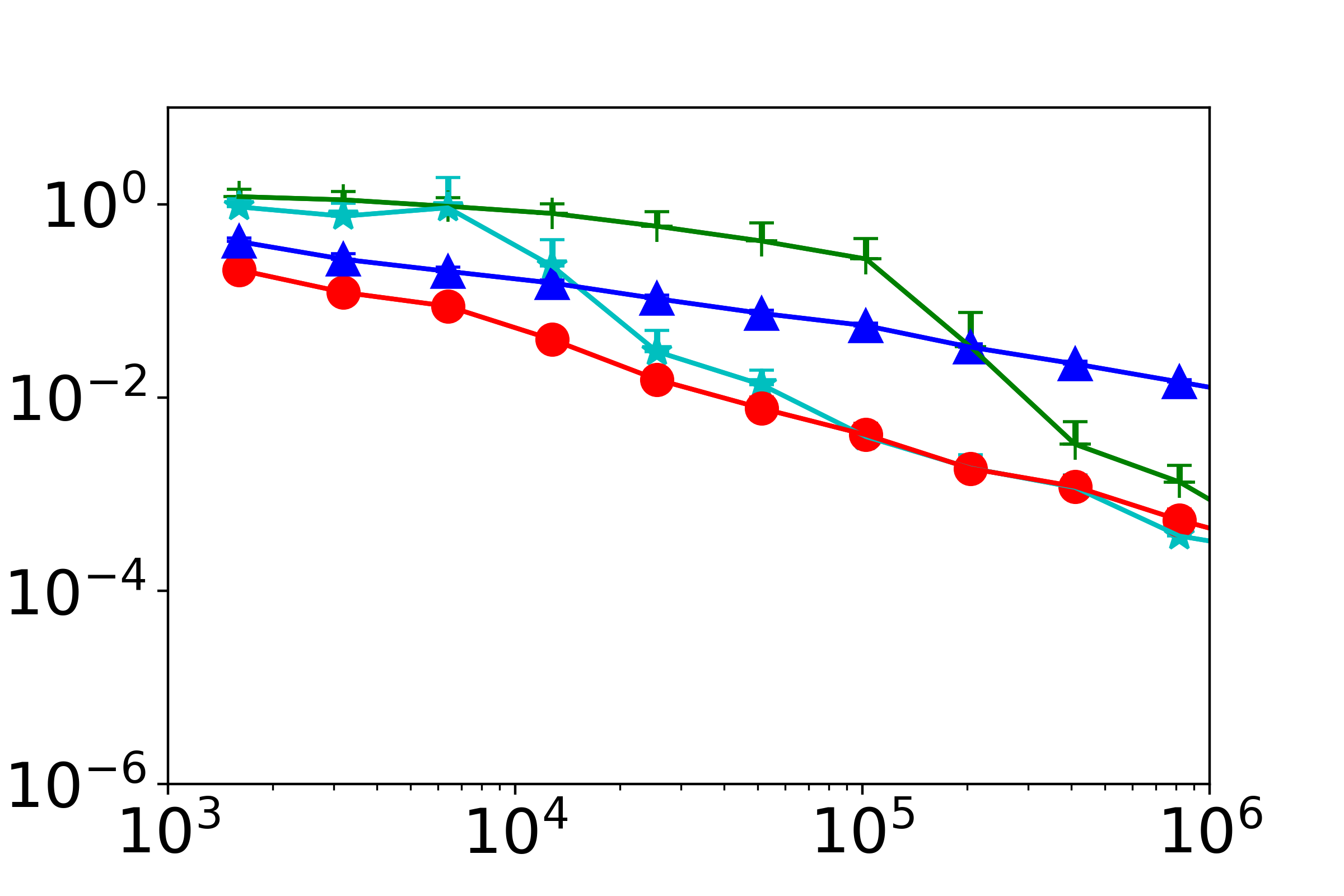}}
\end{minipage}
\begin{minipage}{.33\textwidth}
\subfigure[$\sigma_\eta^2/\sigma_\xi^2=100$, Huber, $\gamma=10$]{
\includegraphics[width=.95\textwidth]{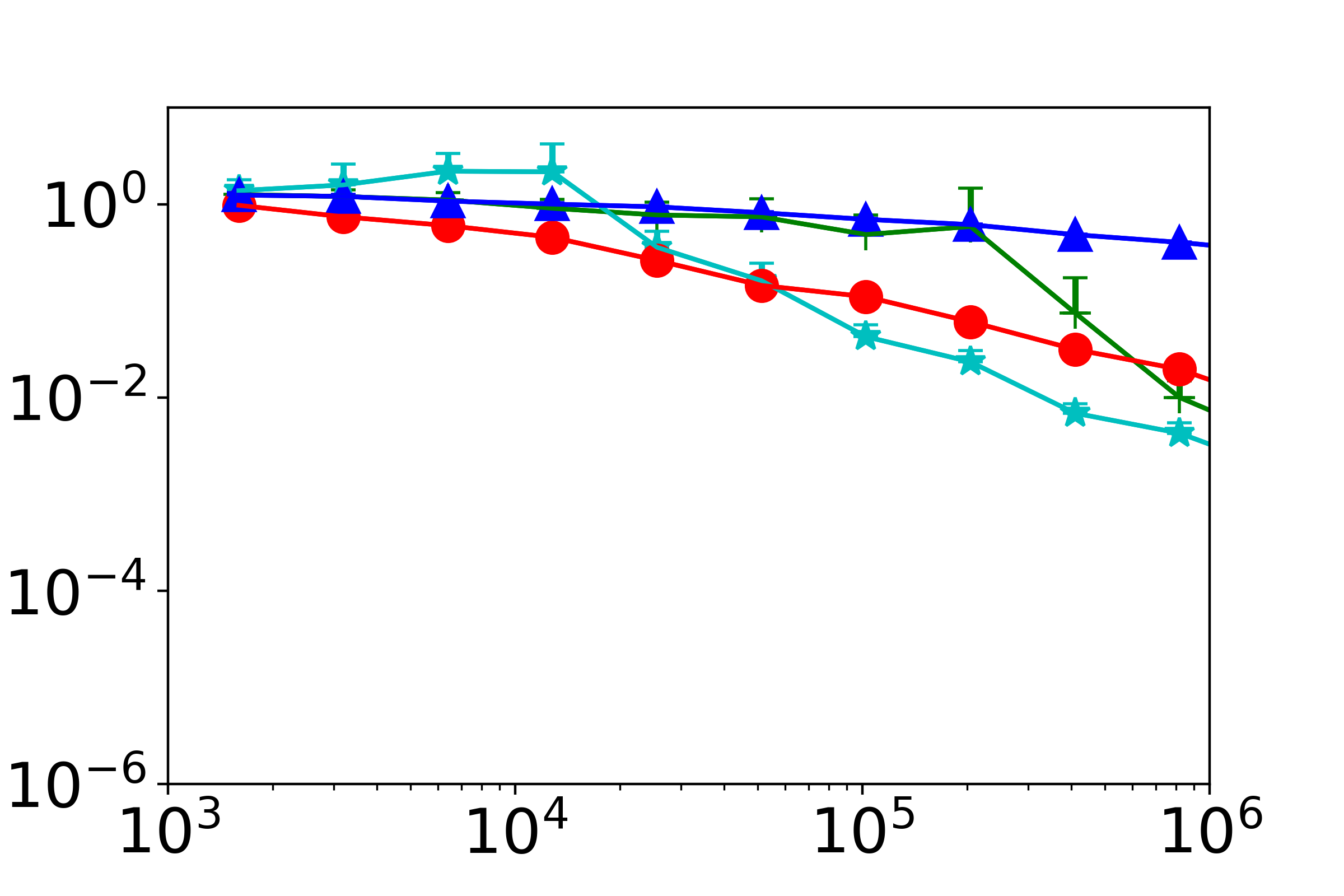}}
\end{minipage}%

\begin{minipage}{1\textwidth}
\centering
\subfigure{
\includegraphics[width=0.8\textwidth]{Legend.png}}
\end{minipage}%
\caption{Error of SAA for absolute value loss and Huber loss, dimension $d=20$}
\end{figure}

Figure~\ref{fig:dependent_absolute} (a)-(c) shows that setting $n=\cO(\sqrt{T})$ indeed yields almost the best accuracy for absolute value loss minimization, which again matches our analysis. The overall performance of SAA for the original and that of the smoothed problems behave quite similarly in this case, yet solving the smoothed problem yields much better accuracy under the same budget. This also supports our theoretical findings that the sample complexity is lower for smooth problems. 

\subsection{Comparison of Conditional Sampling and Independent Sampling}
In this experiment, we consider a modified logistic regression example, that falls into the special case with independent inner and outer randomness: 
\begin{equation*}
\min_{x\in\mathcal{X}} F(x)=\EE_{\xi=(a,b)}\log(1+\exp(-b(\EE_{\eta}\eta+a)^\top x)),
\end{equation*} 
where $a\sim \cN(0,\sigma_\xi^2I_{d})\in\RR^d$ is a random feature vector, $b\in\{\pm 1\}$, $\eta\sim\cN(0,\sigma_\eta^2 I_d)$ is the noise. 
The empirical function of the two sampling schemes $\hat F_{nm}(x)$ is of the form 
\begin{equation*}
    \hat F_{nm}(x) = \frac{1}{n} \sum_{i=1}^n \log\bigg(1+\exp\big(-b_i \big(\frac{1}{m}\sum_{j=1}^m \eta_{ij}+a_i\big)^\top x\big)\bigg).
\end{equation*} 
When employing the independent sampling scheme, we generate $\{\eta_{1j}\}_{j=1}^m$ and let $\eta_{ij} =\eta_{1j}$ for all $i>1$. 

For both sampling schemes, the optimal allocation for $n$ is in the order of $\cO(\sqrt{T})$, and $m$ is set to $m=T/n$ or $m = T-n$. In the experiment, $d=\{10,100\}$. $\sigma_\xi^2 = 1$, and $\sigma_\eta^2=10$, and the samples are generated accordingly. For any given sample budget $T$, we compare the performance of the two sampling scheme under different choices of outer sample $n$ varying from $0$ to $10000$.
 
\begin{figure}[t]
\label{fig:comparison}
\centering
\begin{minipage}{.5\textwidth}
\subfigure[$d=10$, $T=10000$, Varying outer sample size $n$]
{
\includegraphics[width=.95\textwidth]{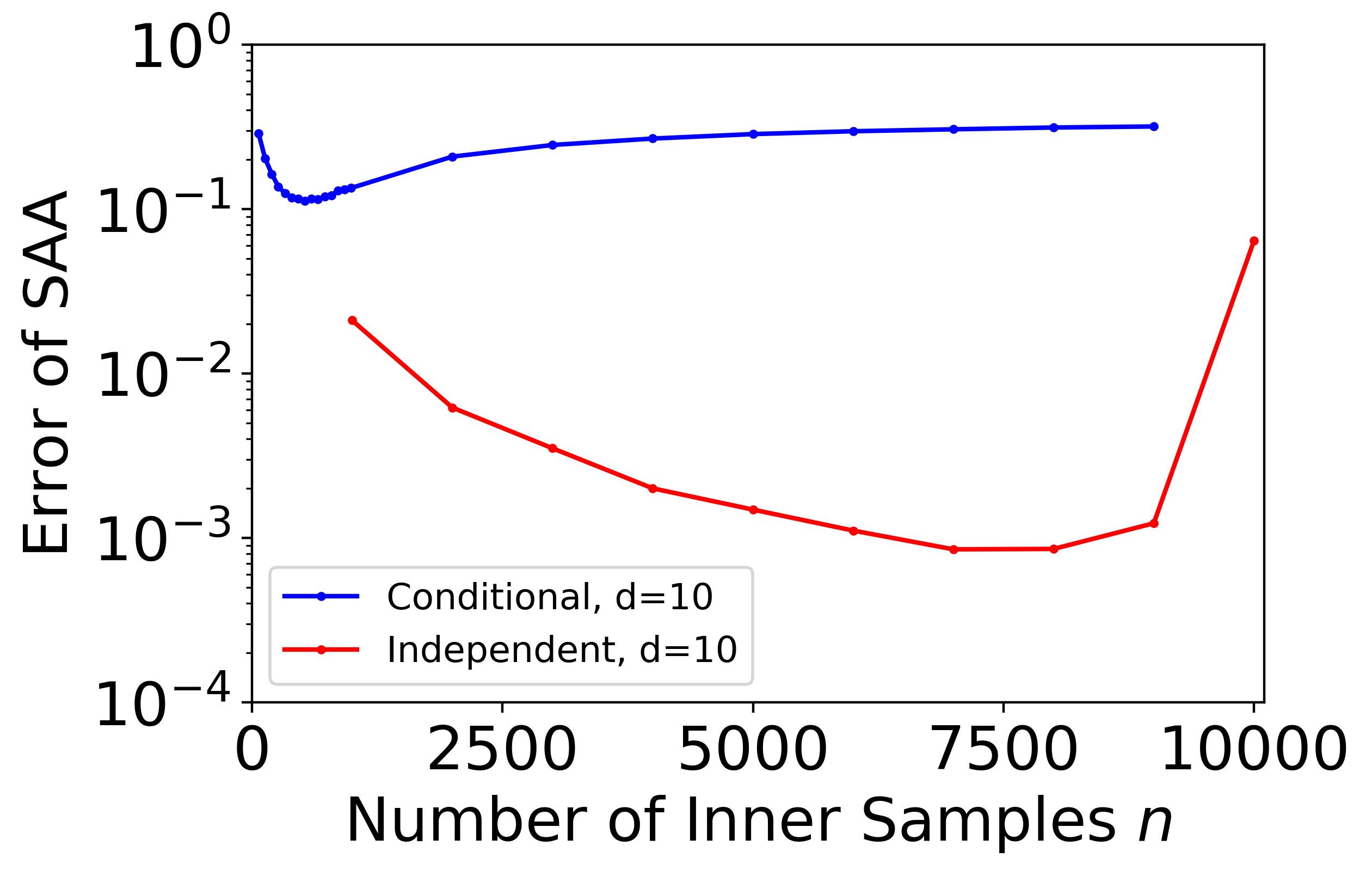}
}
\end{minipage}%
\begin{minipage}{.5\textwidth}
\subfigure[Various $d$ and $T$]
{\includegraphics[width=.95\textwidth]{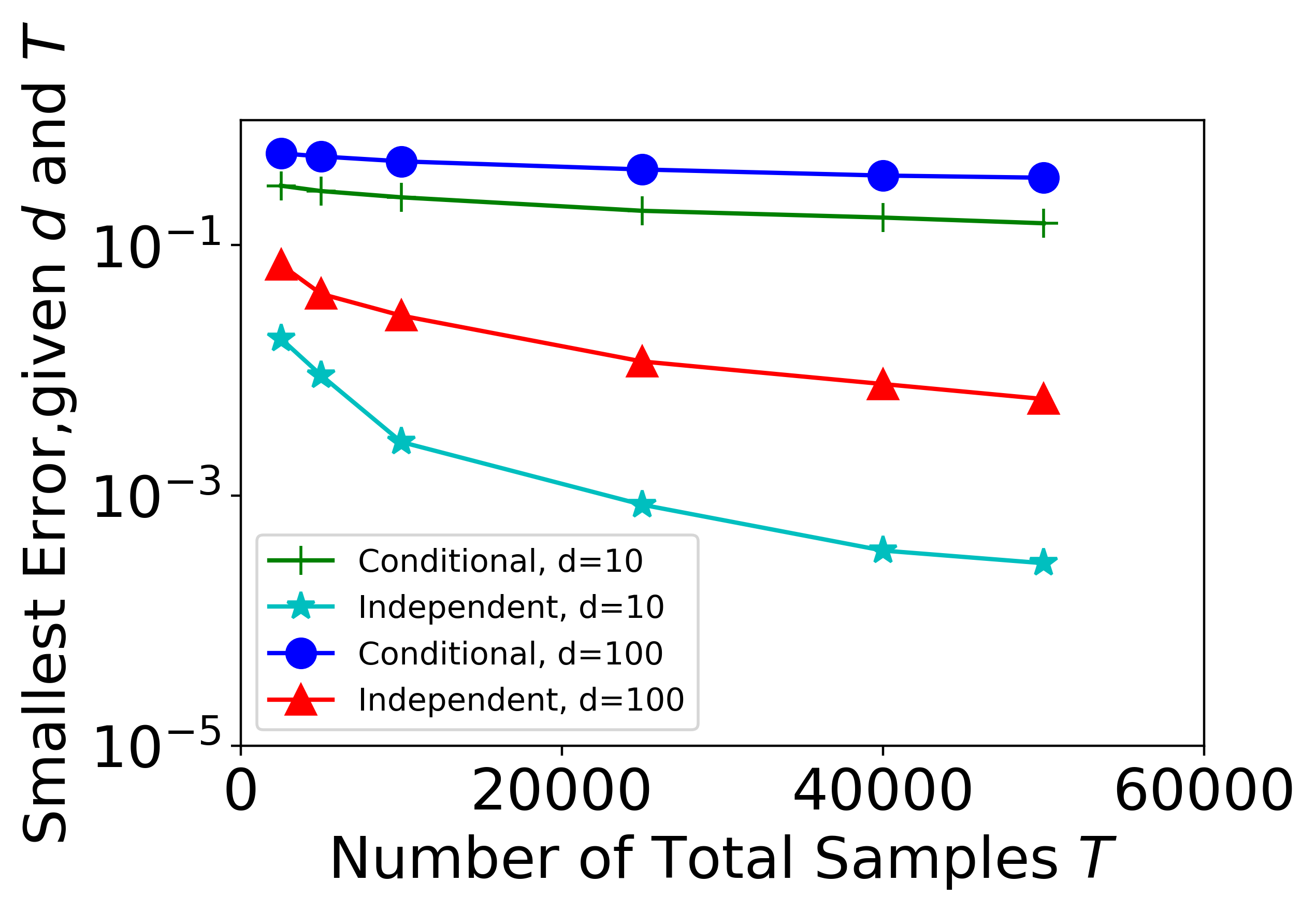}}
\end{minipage}

\caption{Comparison of conditional sampling and independent sampling schemes}
\end{figure}
 
Figure~\ref{fig:comparison}(a) illustrates the comparison when $d=10$, and $ T=10000$. The bell shape in~Figure~\ref{fig:comparison}(a) reflects a clear bias-variance tradeoff for different $n$ and $m$. 

In Figure~\ref{fig:comparison}(b), we report the best performance (by choosing the best $n$) of these two sampling schemes with $d\in\{10, 100\}$, and $T$ ranging from $1000$ to $50000$.   Figure~\ref{fig:comparison}(b) shows that the independent sampling scheme always achieves a smaller error for the logistic regression problem. The gap between  the two schemes decreases as the dimension increases, which also  matches  our analysis.

\section{Conclusion}
\label{sec:con}
In this paper, we introduce the class of conditional stochastic optimization problems and  provide sample complexity analysis of sample average approximation under different structural assumptions. Our results show that the overall sample complexity can be significantly reduced under Lipschitz smoothness condition, which is very different from the theory of classical stochastic optimization and multi-stage stochastic programming. By exploiting error bound conditions, the sample complexity could be further reduced. To our best knowledge, these are the first non-asymptotic sample complexity results established in the context of conditional stochastic optimization. For future work, we will investigate stochastic approximation algorithms for solving this family of problems and establish their sample complexities. 

\appendix
\section{Proof of Propositions}
\label{app:pre}
\subsection{Proof of Lemma \ref{lem:ldb}}
\begin{proof}

The proof of one dimension random variable case was given in~\cite{kleywegt2002sample} using Chernoff bound. Based on that, we consider the case when $X$ is a zero-mean random vector in $\RR^k$. Denote $X_i = (X_i^1,X_i^2,\cdots,X_i^k)^\top$ for $i=1,\cdots,n$, $\sigma_j^2 = \Var(X^j)$, $z_j=\frac{\sum_{j=1}^k \sigma_j^2}{ \sigma_j^2}$, and $I_j(\cdot)$  the rate function of the $j^{th}$ coordinate of the random vector $X$. We have 
\begin{equation}
\label{eq:ldbv}
    \begin{split}
        &\PP(\norm{\bar X }_2 \geq \eps) 
        = \PP\bigg(\sum_{j=1}^k (\bar X^j -\EE X^j)^2\geq \eps^2\bigg)
        \leq \sum_{j=1}^k \PP\bigg( (\bar X^j )^2\geq \frac{\eps^2}{z_j}\bigg)\\
        = &  \sum_{j=1}^k \PP\bigg( |\bar X^j|\geq \frac{\eps}{\sqrt{z_j}}\bigg)
        \leq \sum_{j=1}^k \exp{\bigg(-n \min\bigg\{I_j(\frac{\eps}{\sqrt{z_j}});I_j(-\frac{\eps}{\sqrt{z_j}})\bigg\}\bigg)}
    \end{split}
\end{equation}
By Lemma \ref{lem:ldb} and definition, we get 
\begin{equation*}
    \PP(\norm{\bar X }_2 \geq \eps) \leq 2\sum_{j=1}^k  \exp{\bigg(-\frac{n\eps^2}{(\delta+2)z_j\sigma_j^2}\bigg)}
        =  2k\exp{\bigg(-\frac{n\eps^2}{(\delta+2)\sum_{j=1}^k\sigma_j^2}\bigg)}.
\end{equation*}
Using the fact that $\sum_{j=1}^k\sigma_j^2 \leq \EE \norm{X}_2^2$, we obtain the desired result.
\end{proof}

\section{Proof of Theorem~\ref{thm:iscsaa}}
\lable{sec:appendix-ind}

\paragraph{Convergence Analysis} We follow a similar decomposition as we did in proving Theorem~\ref{thm:SAA_error_bound} and use the same notations, like $\hat F_{nm}^{(k)}(x)$ and $\hat x_{nm}^{(k)}$, the perturbed empirical function and its minimizer, except that we replace all the $\eta_{kj}$ with $\eta_j$ for $k=1,\cdots,n$ and replace the conditional expectation $\EE_{\eta\given\xi}$ with $\EE_{\eta}$. Unfortunately, one will immediately notice that Lemma \ref{lm:bias} is no longer applicable for bounding the second term in (\ref{eq:decompose_term_1}):
\begin{equation*}
\EE \bigg[\frac{1}{n}\sum_{k =1}^n  f_{\xi_k} \bigg( \EE_{\eta} g_\eta(\hat x_{nm}^{(k)},\xi_k) \bigg) -\frac{1}{n}\sum_{k =1}^n f_{\xi_k}\bigg(\frac{1}{m}\sum_{j=1}^m g_{\eta_{j}}(\hat x_{nm}^{(k)},\xi_k)\bigg)\bigg].
\end{equation*}
Because the minimizer $\hat x_{nm}^{(k)}$ depends on $\{\eta_j\}_{j=1}^m$. Then Lemma~\ref{lm:bias} is not applicable. Below we provide the detailed proof of Theorem \ref{thm:iscsaa}.
\begin{proof} 
\if 0
It is clear that $x^*$ has no randomness, $\hat x_{nm}$ is a function of $\{\xi_i\}_{i =1}^n, \{\eta_{j}\}_{j=1}^m$. Since $x^*$ is the minimizer of $F(x)$, $F(\hat x_{nm}) - F(x^*)\geq 0$. By Markov inequality, for any $\eps>0$, 
\begin{equation}
\label{eq:imarkov}
    \PP \bigg( F(\hat x_{nm}) - F(x^*) \geq \eps \bigg)
    \leq \frac{\EE [F(\hat x_{nm}) - F(x^*)]}{\eps}  .
\end{equation}

We decompose the error $F(\hat x_{nm}) - F(x^*)$
\begin{equation*}
\begin{split}
    F(\hat x_{nm}) - F(x^*)
    =
    F(\hat x_{nm})- \hat F_{nm}(\hat x_{nm}) + \hat F_{nm}(\hat x_{nm}) - \hat F_{nm}(x^*)+ \hat F_{nm}(x^*)-  F(x^*)\\
\end{split}
\end{equation*}
\fi

Define $\Epsilon_1 := F(\hat x_{nm})- \hat F_{nm}(\hat x_{nm})$, and
\begin{equation*}
     \hat F_{nm}^{(k)}(x) := \frac{1}{n}\sum_{i\not = k}^n f_{\xi_i}\bigg(\frac{1}{m}\sum_{j =1}^m g_{\eta_{j}}(x,\xi_i)\bigg)+\frac{1}{n}f_{\xi_k^\prime}\bigg(\frac{1}{m}\sum_{j =1}^m g_{\eta_{j}}(x,\xi_k^\prime)\bigg),
\end{equation*}
the empirical function by replacing the outer sample $\xi_k$ with an i.i.d sample $\xi_k^{\prime}$. Denote $\hat x_{nm}^{(k)} = \argmin_{x \in \mathcal{X}} \hat F_{nm}^{(k)}(x)$. Then, $\EE \Epsilon_1 $ could be written as:
\begin{equation}\label{eq:decompose_term_2}
\begin{split}
\EE \Epsilon_1 = &\EE\bigg[F(\hat x_{nm})
-\frac{1}{n}\sum_{k =1}^n  f_{\xi_k} \bigg( \EE_{\eta} g_\eta(\hat x_{nm}^{(k)},\xi_k) \bigg)\bigg]\\
+& \EE \bigg[\frac{1}{n}\sum_{k=1}^n f_{\xi_k} \bigg( \EE_{\eta} g_\eta(\hat x_{nm}^{(k)},\xi_k) \bigg) - \frac{1}{n}\sum_{k=1}^n f_{\xi_k}\bigg(\frac{1}{m}\sum_{j=1}^m g_{\eta_{j}}(\hat x_{nm}^{(k)},\xi_k)\bigg)\bigg]\\
+&\EE \bigg[\frac{1}{n}\sum_{k =1}^n f_{\xi_k}\bigg(\frac{1}{m}\sum_{j=1}^m g_{\eta_{j}}(\hat x_{nm}^{(k)},\xi_k)\bigg) - \hat F_{nm}(\hat x_{nm})\bigg].
\end{split}
\end{equation}

Since $\xi_k$ and $\xi_k^\prime$ are i.i.d,  $\hat x_{nm}$ and $\hat x_{nm}^{(k)}$ follow identical distribution. Then $\EE F(\hat x_{nm}) = \EE F(\hat x_{nm}^{(k)})$. As $\hat x_{nm}^{(k)}$ is independent of $\xi_k$, by definition of $F(x)$, we know $\EE F(\hat x_{nm}^{(k)})=\EE f_{\xi_k} ( \EE_{\eta} g_\eta(\hat x_{nm}^{(k)},\xi_k))$ for any $k=1,\cdots,n$. As a result, the first term is $0$. 

To analyze the second term, denote 
\begin{equation*} H_k(x):= 
f_{\xi_k} \bigg( \EE_{\eta} g_\eta(x,\xi_k) \bigg) - f_{\xi_k}\bigg(\frac{1}{m}\sum_{j=1}^m g_{\eta_{j}}(x,\xi_k)\bigg).
\end{equation*}
We pick a $\upsilon$-net $\{x_l\}_{l=1}^{Q}$ for the decision set $\mathcal{X}$, such that for any $x\in\mathcal{X}$, there exists $l_0 \in \{1,\cdots, Q\}$, $\|x-x_{l_0}\|\leq \upsilon$. Then it holds for any $s>0$, 
\begin{equation}
\label{eq:expmax}
\begin{split}
      & \exp\left(s\EE H_k(\hat x_{nm}^{(k)})\right)  
      \leq   \exp\left(  s\EE\max_{l=1,\cdots,Q} H_k(x_l)+2s\upsilon L_f L_g\right) \\
      \leq &\EE \exp\left(  s\max_{l=1,\cdots,Q} H_k(x_l)+2s\upsilon L_f L_g\right)  
      = \EE \max_{l=1,\cdots,Q}\exp\left(   sH_k(x_l)+2s\upsilon L_f L_g\right)\\
       \leq & \EE \sum_{l=1}^Q \exp\left( sH_k(x_l)+2s\upsilon L_f L_g\right) 
       = \sum_{l=1}^Q \EE \exp\left( sH_k(x_l)+2s\upsilon L_f L_g\right) .
\end{split}
\end{equation}
The first inequality holds as $\hat x_{nm}^{(k)}$ is independent of $\xi_k$, and  $f_{\xi}(\cdot)$ and $g_\eta(\cdot, \xi)$ are Lipschitz continuous, which implies,
\begin{equation*}
H_k(\hat x_{nm}^{(k)}) \leq \sup_{x\in\mathcal{X}} H_k(x)  \leq \max_{l=1,\cdots,Q} H_k(x_l)+2\upsilon L_f L_g.
\end{equation*} 
The second inequality holds by Jensen's inequality.
Next we show that $H_k(x_l)-\EE H_k(x_l)$ is a sub-Gaussian random variable for any given $\xi_k$. Since  $H_k(x_l)$ is a function of $\{\eta_j\}_{j=1}^m$. Denote $H_k(x_l) := \tilde H(\eta_1, \ldots, \eta_m)$. Then for any $p\in[m]$, and given $\eta_1,\ldots,\eta_{p-1}, \eta_{p+1},$ $\cdots, \eta_m$, we have
\begin{equation*}
\begin{split}
    & \sup_{\eta_p^\prime} \tilde H(\eta_1,\cdots,\eta_{p-1},\eta_p^\prime, \eta_{p+1},\cdots, \eta_m) - \inf_{\eta_p^{\prime\prime}} \tilde H(\eta_1,\cdots,\eta_{p-1},\eta_p^{\prime\prime}, \eta_{p+1},\cdots, \eta_m)\\
    = & \sup_{\eta_p^\prime, \eta_p^{\prime\prime}} \EE_{\xi_k} f_{\xi_k}\!\bigg(\!\frac{1}{m}\sum_{j\not=p}^m g_{\eta_{j}}(x,\xi_k)\!+\!\frac{1}{m}g_{\eta_{p}^{\prime\prime}}(x,\xi_k)\bigg)\!-\!f_{\xi_k}\!\bigg(\!\frac{1}{m}\sum_{j\not=p}^m g_{\eta_{j}}(x,\xi_k)\!+\!\frac{1}{m}g_{\eta_{p}^{\prime}}(x,\xi_k)\!\bigg)\\
    \leq & \sup_{\eta_p^\prime, \eta_p^{\prime\prime}}\EE_{\xi_k} \frac{L_f}{m}\bigg|g_{\eta_{p}^{\prime\prime}}(x,\xi_k)-g_{\eta_{p}^{\prime}}(x,\xi_k)\bigg| \\
    \leq & \frac{2M_gL_f}{m},
\end{split}
\end{equation*}
where $M_g$ is the upper bound of $g_\eta(\cdot, \xi)$ on $\mathcal{X}$.
It implies that $H_k(x_l)=\tilde H(\eta_1, \cdots, \eta_m)$ has bounded difference $\frac{2M_gL_f}{m}$.
By McDiarmid’s inequality~\cite{mcdiarmid_1989}, for any $r>0$,
\begin{equation*}
    \PP( H_k(x_l) -\EE H_k(x_l) \geq r) \leq 2\exp\bigg(-\frac{r^2m}{2M_g^2 L_g^2}\bigg).
\end{equation*}
It implies that $H_k(x_l)-\EE H_k(x_l)$ is a sub-Gaussian random variable with zero mean and variance proxy $2M_g^2L_f^2/m$ for any given $\xi_k$.
By definition it yields
\begin{equation*}
    \EE \exp\left(s\left[H_k(x_l)-\EE H_k(x_l)\right]\right) \leq \exp\bigg(\frac{2M_g^2L_f^2s^2 }{m}\bigg).
\end{equation*}
Since $x_l$ is independent of random vectors $\{\eta_j\}_{j=1}^m $, by Lemma~\ref{lm:bias}, we know $\EE H_k(x_l)\leq \frac{L_f\sigma_g}{\sqrt{m}}$. It further implies
\begin{equation*}
\EE \exp(sH_k(x_l)) \leq \exp\bigg(\frac{2M_g^2L_f^2s^2 }{m}+ \frac{sL_f\sigma_g}{\sqrt{m}} \bigg).
\end{equation*}
With ~(\ref{eq:expmax}), we have
\begin{equation*}
     \exp\left(s\EE H_k(\hat x_{nm}^{(k)})\right) 
     \leq
     Q\exp\left(\frac{2M_g^2L_f^2s^2 }{m}+\frac{sL_f\sigma_g}{\sqrt{m}}+2s\upsilon L_f L_g\right).
     \nonumber
\end{equation*}
Taking the logarithm, dividing $s$ on each side, and minimizing over $s$ yields
\begin{equation}
    \EE H_k(\hat x_{nm}^{(k)}) 
    \leq 
    2\sqrt{\frac{2\log(Q) L_f^2 M_g^2}{m}}+\frac{L_f\sigma_g}{\sqrt{m}}+2\upsilon L_f L_g.
    \nonumber
\end{equation}
Since $Q \leq \cO(1)(D_\mathcal{X}/\upsilon)^d$, we have
\begin{equation}
\label{eq:isca1}
\begin{split}
\EE H_k(\hat x_{nm}^{(k)}) & \leq \cO(1)\frac{L_fM_g}{\sqrt{m}}  \sqrt{d\log\left(\frac{D_\mathcal{X} }{\upsilon}\right)} +\frac{L_f\sigma_g}{\sqrt{m}}+2\upsilon L_f L_g.
\end{split}
\end{equation}

For the third term in (\ref{eq:decompose_term_2}), by following the similar steps from (\ref{eq:sca1}) to (\ref{eq:strong2}), we obtain
\begin{equation}
\label{eq:isca3}
\begin{split}
    & \EE \bigg[\frac{1}{n}\sum_{k =1}^n f_{\xi_k}\bigg(\frac{1}{m}\sum_{j=1}^m g_{\eta_{j}}(\hat x_{nm}^{(k)},\xi_k)\bigg) - \hat F_{nm}(\hat x_{nm})\bigg] 
    \leq  L_f L_g\bigg(\frac{2L_f L_g}{\mu n}\bigg)^{1/\delta}.
\end{split}
\end{equation}
Combining with (\ref{eq:decompose_term_2}), (\ref{eq:isca1}), and (\ref{eq:isca3}),
\begin{equation}
\label{eq:iscA}
    \EE \Epsilon_1
    \leq 
    L_f L_g\bigg(\frac{2L_f L_g}{\mu n}\bigg)^{1/\delta}+\cO(1)\frac{L_fM_g}{\sqrt{m}}  \sqrt{d\log\left(\frac{12D_\mathcal{X}L_fL_g}{\alpha\eps}\right)} +\frac{L_f\sigma_g}{\sqrt{m}}+2\upsilon L_f L_g.
\end{equation}
Similar with the steps from (\ref{eq:scB}) and (\ref{eq:scC}), by optimality of $\hat x_{nm}$ of $\hat F_{nm}$ and Lemma \ref{lm:bias},
\begin{equation}
\label{eq:iscB}
    \EE [\hat F_{nm}(\hat x_{nm})- \hat F_{nm}(x^*)] \leq 0; \quad \EE [ \hat F_{nm}(x^*)- F(x^*)]
    \leq 
    \frac{L_{f}\sigma_{g}}{\sqrt{m}}.
\end{equation}
Finally, combining (\ref{eq:iscA}), (\ref{eq:iscB}), with Markov inequality, we obtain (\ref{eq:convergence_iscsaa}).
\if 0
, namely,
\begin{equation*}
\begin{split}
&\PP (F(\hat x_{nm}) - F(x^*) \geq \epsilon )\\
    \leq &
    L_f L_g\bigg(\frac{2L_f L_g}{\mu n\eps^\delta}\bigg)^{1/\delta}
    +
    \cO(1)\frac{L_fM_g}{\sqrt{m}}  \sqrt{d\log\left(\frac{12D_\mathcal{X}L_fL_g}{\alpha\eps}\right)} +\frac{2L_f\sigma_g}{\sqrt{m\eps^2}}+2\upsilon L_f L_g.
\end{split}
\end{equation*}
\fi

Let
\begin{equation*}
L_f L_g\bigg(\frac{2L_f L_g}{\mu n\eps^\delta}\bigg)^{1/\delta}\leq \frac{\alpha}{2}; \quad  \cO(1)\frac{L_fM_g}{\sqrt{m\eps^2}}  \sqrt{d\log\left(\frac{D_\mathcal{X} }{\upsilon}\right)}\leq \frac{\alpha}{6}; \quad \frac{2L_f\sigma_g}{\sqrt{m\eps^2}} \leq \frac{\alpha}{6}.
\end{equation*}
We obtain the desired sample complexity~(\ref{eq:isamplecomplexity}).
\end{proof}

\section{Example of Huber Loss Minimization}\label{sec:appendix-ex}
\if 0
For $\gamma >0$, consider the following problem
\begin{equation*}
\min_{x\in\mathcal{X}} F(x):=  H(\EE_\eta [x+\eta],\gamma)+(\EE_\eta [x+\eta])^2,
\end{equation*}
where $\eta\sim\cN(0,\sigma_\eta^2)$ and $H(\cdot,\gamma)$ is the Huber function, i.e., 
\begin{equation*}
H(x,\gamma) = 
\left\{
\begin{aligned}
& |x|-\frac{1}{2}\gamma && \text{for }{|x|>\gamma},\\
&\frac{1}{2\gamma}x^2 &&  \text{for }{|x|\leq\gamma}.
\end{aligned}
\right.
\end{equation*}
Note that here $f(x) = H(x,\gamma)+x^2$ is deterministic, $g_\eta(x,\xi) = x+\eta$. When $\gamma >0$, $f(x)$ is $1/\gamma$-Lipschitz smooth. When $\gamma\rightarrow 0$, $f(x) \rightarrow |x|+x^2$, which is no longer differentiable. In this simple example, $x^* = \argmin_{x\in\mathcal{X}}F(x) = -\EE \eta$, $F^* = \min_{x\in\mathcal{X}}F(x) =  0$. The empirical objective function becomes
$
\hat F_{m}(x) = H(x+\bar \eta, \gamma)+(x+\bar \eta)^2,
$
where $\bar \eta =  \frac{1}{m}\sum_{j=1}^m \eta_j$. Then $\hat x_{m} =\argmin_{x\in\mathcal{X}}\hat F_{m}(x) =  -\bar \eta$.
\fi

To show (\ref{example3}), denote $Y = \EE \eta - \bar \eta$, then $Y \sim \cN(0, \frac{\sigma_\eta^2}{m})$. Then the error of SAA is 
\begin{equation}
\begin{split}
\label{eq:example3_2}
&\EE F(\hat x_{m}) - F(x^*)  = \EE H(\EE \eta - \bar \eta, \gamma)+ \EE (\bar \eta - \EE \eta)^2\\ 
= &  \int_{0}^\gamma \frac{1}{\gamma}y^2 p(y) dy +2\int^{+\infty}_\gamma \bigg(y-\frac{1}{2}\gamma\bigg) p(y) dy  +\EE Y^2,
\end{split}
\end{equation}
where $p(y) = \frac{\sqrt{m}}{\sqrt{2\pi\sigma_\eta^2}} \exp\big(-\frac{my^2}{2\sigma_\eta^2}\big)$ is the PDF of $Y$, and $\EE Y^2 = \frac{\sigma_\eta^2}{m}$.
Denote erf$(x) := \frac{2}{\sqrt{\pi}}\int_{0}^x \exp(-x^2)dx$, $y_1 := y \sqrt{\frac{m}{2\sigma_\eta^2}}$. The first term in (\ref{eq:example3_2}) is
\begin{equation*}
\begin{split}
\int_{0}^\gamma\! \frac{1}{\gamma}y^2 p(y) dy
= &
\frac{2\sigma_\eta^2}{m\gamma\sqrt{\pi}}\int_{0}^{\gamma\sqrt{\frac{m}{2\sigma_\eta^2}}} y_1^2  \exp(-y_1^2) dy_1\\
= &
\frac{\sigma_\eta^2}{2\gamma m} \text{erf}\bigg(\sqrt{\frac{\gamma^2 m}{2\sigma_\eta^2}}\bigg)-\sqrt{\frac{\sigma_\eta^2}{2\pi m}}\exp\bigg(-\frac{\gamma^2 m }{2\sigma_\eta^2}\bigg).
\end{split}
\end{equation*}
We use the fact that:
\begin{equation*}
\int_{0}^z x^2\exp(-x^2) dx = \frac{1}{4}\sqrt{\pi}\text{erf}(z) - \frac{1}{2}\exp(-z^2)z.
\end{equation*}
The second term in (\ref{eq:example3_2}) is bounded by 
\begin{equation*}
\sqrt{\frac{\sigma_\eta^2}{2\pi m}}\exp\bigg(-\frac{m\gamma^2}{2\sigma_\eta^2}\bigg)=\int^{+\infty}_\gamma y p(y) dy \leq 2 \int^{+\infty}_\gamma (y-\frac{1}{2}\gamma) p(y) dy \leq 2\int^{+\infty}_\gamma y p(y) dy.
\nonumber
\end{equation*}
Combining them together, we have (\ref{example3}).
\if 0
\begin{equation*}
\begin{split}
0\leq \EE F(\hat x_{nm}) - F(x^*)-\bigg(\frac{\sigma_\eta^2}{2\gamma m} \text{erf}(\sqrt{\frac{\gamma^2 m}{2\sigma_\eta^2}})+\frac{\sigma_\eta^2}{m}\bigg) \leq  \sqrt{\frac{\sigma_\eta^2}{2\pi m}}\exp(-\frac{m\gamma^2}{2\sigma_\eta^2}) .
\end{split}
\end{equation*}
\fi
For a given $\gamma>0$,
erf$\big(\sqrt{\frac{\gamma^2 m}{2\sigma_\eta^2}}\big) \rightarrow 1$ as $m\rightarrow\infty$. By (\ref{example3}), we have
\begin{equation*}
\EE F(\hat x_{nm}) - F(x^*) 
= \cO\bigg(\frac{1}{m}\bigg).
\end{equation*}
When $\gamma \rightarrow 0$, (\ref{eq:example3_2}) becomes
\begin{equation*}
\begin{split}
\lim_{\gamma\rightarrow 0}\EE F(\hat x_{m}) - F(x^*) &=\lim_{\gamma\rightarrow 0} \int_{0}^\gamma \frac{1}{\gamma}y^2 p(y) dy  +2\int^{+\infty}_\gamma \bigg(y-\frac{1}{2}\gamma\bigg) p(y) dy  + \frac{\sigma_\eta^2}{m}\\
& = \sqrt{\frac{\sigma_\eta^2}{2\pi m}}+\frac{\sigma_\eta^2}{m} = \cO\bigg(\frac{1}{\sqrt{m}}\bigg).
\end{split}
\end{equation*}

\section{Empirical Objectives  Satisfying Quadratic Growth Condition}
\label{app:qg}

\paragraph{Strongly Convex Function Composed with Linear Function}
The empirical objective function is $\hat F_{nm}(x) = \frac{1}{n}\sum_{i=1}^n f_{\xi_i}(A_i x)$, where $f_\xi(\cdot)$ is $\mu$-strongly convex, $A_i x:=\frac{1}{m} \sum_{j=1}^m g_{\eta_{ij}}(x,\xi_i) $, the average of  linear inner function $g_{\eta_{ij}}(x,\xi_i):=A_{\eta_{ij}} x$. 
To show that $\hat F_{nm}(x)$ satisfies the QG condition.  Denote $u_i = A_i y$, $v_i = A_i x$. Since $f_\xi(\cdot)$ is strongly convex,
\begin{equation*}
f_{\xi_i}(u_i) - f_{\xi_i}(v_i)- \nabla f_{\xi_i}(v_i)^\top (u_i-v_i) \geq \frac{\mu}{2}\norm{u_i-v_i}_2^2.
\end{equation*}
Taking average over $n$ such inequalities, we obtain
\begin{equation*}
\frac{1}{n}\sum_{i=1}^n f_{\xi_i}(u_i) - f_{\xi_i}(v_i)- \nabla f_{\xi_i}(v_i)^\top (u_i-v_i) \geq \frac{1}{n}\sum_{i=1}^n\frac{\mu}{2}\norm{u_i-v_i}_2^2.
\end{equation*}
Replacing $u_i$, $v_i$ with $A_i y$ and $A_i x$, we have
\begin{equation*}
\frac{1}{n}\sum_{i=1}^n f_{\xi_i}(A_i y) - f_{\xi_i}(A_i x)- \nabla f_{\xi_i}(A_i x)^\top A_i (y- x) \geq \frac{1}{n}\sum_{i=1}^n\frac{\mu}{2}(y- x)^\top A_i^\top A_i (y-x).
\end{equation*}
Since $\nabla \hat F_{nm}(x)^\top =  \frac{1}{n}\sum_{i=1}^n (A_i^\top \nabla f_{\xi_i}(A_i x))^\top = \frac{1}{n}\sum_{i=1}^n  \nabla f_{\xi_i}(A_i x)^\top A_i$, we get
\begin{equation*}
\hat F_{nm}(y) - \hat F_{nm}(x)- \nabla \hat F_{nm}(x)^\top (y-x) \geq \frac{1}{n}\sum_{i=1}^n\frac{\mu}{2}\norm{A_i (y- x)}_2^2 \geq  \frac{\mu}{2} \norm{\frac{1}{n}\sum_{i=1}^n A_i(y-x)}_2^2.
\end{equation*}
Let $z$ be a point in $\mathcal{X}^*$, we have
\begin{equation}
\label{eq:satisfyQG}
\begin{split}
    \hat F_{nm}(x) - \hat F_{nm}(z)\geq  \frac{\mu}{2} \norm{\frac{1}{n}\sum_{i=1}^n A_i(x-z)}_2^2 &\geq \frac{\mu\theta(\frac{1}{n}\sum_{i=1}^n A_i)}{2}\norm{x-z}_2^2 \\
    &\geq \min_{z\in\mathcal{X}^*}\frac{\mu\theta(\frac{1}{n}\sum_{i=1}^n A_i)}{2}\norm{x-z}_2^2.
    \end{split}
\end{equation}
Here $\theta(A)$ is the smallest non-zero singular of $A$. Thus $\hat F_{nm}(x)$ satisfies quadratic growth condition for any $n$ and $m$. A special case is when $n=m=1$, i.e., a strongly convex objective composed with  a linear function satisfies QG condition.

\paragraph{Some Strictly Convex Functions Composed with Linear Function on a Compact Set} Consider Example 2, the logistic regression problem with the objective 
\begin{equation*}
F(x)=\EE_{\xi=(a,b)}\log(1+\exp(-b\EE_{\eta|\xi}[\eta]^Tx)),
\end{equation*} 
where $a\in\RR^d$ is a random feature vector and $b\in\{1,-1\}$ is the label, $\eta=a+\cN(0,\sigma^2 I_d)$ is a perturbed noisy observation of the input feature vector $a$. Its empirical objective function $\hat F_{nm}(x)$ is given by
\begin{equation*}
    \hat F_{nm} (x) = \frac{1}{n} \sum_{i=1}^n \log\bigg(1+\exp\bigg(-b_i \frac{1}{m}\sum_{j=1}^m \eta_{ij}^\top x\bigg)\bigg).
\end{equation*}
where $\EE \eta_{ij} = a_i$.
Here $f_{\xi_i}(u) = \log\big(1+\exp(b_i u)\big)$. $\hat F_{nm}(x) = 1/n \sum_{i=1}^n f(u_i)$, where $f(u)= \log\big(1+\exp(u)\big)$ is strictly convex, and $u_i = \frac{1}{m} \sum_{j=1}^m \eta_{ij}^\top x$ is bounded for any $x\in\mathcal{X}$ and realization $\eta_{ij}$. It is easy to verify that on any compact set, $f(u)$ is strongly convex. The strong convexity parameter is related to the compact set. With~(\ref{eq:satisfyQG}), $\hat F_{nm}(x)$ satisfies the QG condition.

Note that the result is not necessarily true for all strictly convex function. For instance, $\norm{x}_2^4$ is strictly convex, but $\norm{Ax}_2^4$ does not satisfies quadratic growth condition on any compact set containing $x=0$.

\section{Other Results on Regularized SAA}
\label{app:regularied}

The Theorem~\ref{thm:SAA_error_bound} discuss the sample complexity of SAA for strongly convex and QG condition cases. We show that the result obtained in Theorem~\ref{thm:SAA_error_bound} can be used to obtain dimensional free sample complexity for general convex objective by adding $l_2$-regularization.

\begin{lm}[\cite{shalev2010learnability}]
\label{lm:sctc}
Consider a stochastic convex optimization problem: 
\begin{equation*}
\min_{x\in\mathcal{X}} G(x),
\end{equation*}
where $G(x)$ is the expectation over some convex random function.
Suppose that the decision set $\mathcal{X}\in \RR^d$ has bounded diameter $D_\mathcal{X}$. Denote $G_\mu(x) := G(x)+\frac{\mu}{2}\norm{x}_2^2$, where $\mu>0$ is a strongly convex parameter. Denote $\hat G(x)$ as the SAA counterpart of $G(x)$, $x^*\in \argmin_{x \in \mathcal{X}} G(x)$, $\hat x \in \argmin_{x \in \mathcal{X}} \hat G(x)$, $x^*_\mu = \argmin_{x \in \mathcal{X}} G_\mu(x)$, and $\hat x_u$ the minimizer of SAA of the regularized objective, namely $\hat x_u= \argmin_{x \in \mathcal{X}} \hat G_\mu(x):=\hat G(x)+ \frac{\mu}{2}\norm{x}_2^2$.
If $\EE [G_{\mu}(\hat x_{\mu})-G_{\mu}(x^*_{\mu})] \leq \beta(\mu)$, then
\begin{equation*}
\EE [G(\hat x_\mu)-G(x^*)]  \leq \beta(\mu)+\frac{\mu}{2}D_\mathcal{X}^2.
\end{equation*}
\end{lm}
\begin{remark}
This theorem shows that the minimum point $\hat x_\mu$ to a $l_2$-regularized empirical function $\hat G_\mu$ could be a good solution to the original convex function $G(x)$ as long as one selects $\mu$ properly. Note that $\hat x_\mu$ might not be a minimum point of the empirical function $\hat G(x)$. In CSO case, according to Theorem \ref{thm:SAA_error_bound}, if $F(x)$ is convex, the expected error of SAA method for $\min_{x\in \mathcal{X}} F(x)+\frac{\mu}{2}\norm{x}_2^2$ is bounded by $\beta(\mu) = \frac{4L_f^2 L_g^2}{\mu n}+ 2\Delta(m)
$. Then,
$
\EE F(\hat x_{nm}) - F(x^*) \leq \frac{4L_f^2 L_g^2}{\mu n}+\frac{\mu}{2}D_\mathcal{X}^2+2\Delta(m).
$
Minimizing over $\mu$, and by Markov inequality, we obtain,
\begin{equation*}
 \PP(F(\hat x_{nm}) - F(x^*)\geq \eps)\leq \frac{2\sqrt{2}L_f L_g D_\mathcal{X}}{\sqrt{n\eps^2} }+\frac{2\Delta(m)}{\eps}.   
\end{equation*}
We notice that the outer sample size, $n=\cO(1/\eps^2)$,  is  dimensional free, while in Theorem \ref{thm:saa}, $n=\cO(d/\eps^2)$, depends linearly in dimension; the inner sample size $m$ is not affected. For high-dimensional problems, adding regularization is sometimes more favorable as it lowers the sample complexity by $d$ and also helps boosting the convergence when solving the SAA. 
\end{remark}
\if 0 
\begin{proof}
\begin{equation*}
\begin{aligned}
~& \beta(\mu) \geq\EE G_\mu(\hat x_\mu)-G_\mu(x^*_\mu) 
 \geq \EE G_\mu(\hat x_\mu)-G_\mu(x^*)\\
 = &\EE G(\hat x_\mu) +\frac{\mu}{2}\norm{\hat x}_2^2- (G(x^*)  +\frac{\mu}{2}\norm{x^*}_2^2)
\geq \EE G(\hat x_\mu) - (G(x^*)  +\frac{\mu}{2}\norm{x^*}_2^2)\\
 \geq &\EE G(\hat x_\mu) - G(x^*)-\frac{\mu}{2}D_\mathcal{X}^2.
\end{aligned}
\end{equation*}
The first inequality is by assumption on $\beta(\mu)$. Switching side we get the desired result.
\end{proof}
\fi

\section*{Acknowledgments}
We would like to acknowledge Alexander Shapiro and Lin Xiao for fruitful discussions and the reviewers for their helpful comments.

\bibliographystyle{siamplain.bst}
\bibliography{ref.bib}
\end{document}